%% file: ex_article.tex
\documentclass[onefignum,onetabnum]{siamart171218}

\input{ex_shared}

\ifpdf
\hypersetup{
  pdftitle={Randomly sparsified Richardson iteration: A dimension-independent sparse linear solver},
  pdfauthor={J. Weare and R. J. Webber}
}
\fi

\begin{document}

\maketitle

\begin{abstract}
Recently, a class of algorithms combining classical fixed point iterations with repeated random sparsification of approximate solution vectors has been successfully applied to eigenproblems with matrices as large as $10^{108} \times 10^{108}$. So far, a complete mathematical explanation for this success has proven elusive. 

The family of methods has not yet been extended to the important case of linear system solves. In this paper we propose a new scheme based on repeated random sparsification that is capable of solving sparse linear systems in arbitrarily high dimensions.
We provide a complete mathematical analysis of this new algorithm.
Our analysis establishes a faster-than-Monte Carlo convergence rate
and justifies use 
of the scheme even when the solution vector itself is too large to store.
\end{abstract}



\section{Introduction}

In this paper, we propose a randomized approach for solving a linear systems of equations
\begin{equation*}
    \bm{A} \bm{x} = \bm{b},
\end{equation*}
involving a square matrix $\bm{A} \in \mathbb{C}^{n \times n}$, which is typically nonsymmetric, and a vector $\bm{b} \in \mathbb{C}^n$.
Our new approach combines Richardson fixed-point iteration \cite{richardson1911approximation} with a strategy of random sparsification.
The algorithm only requires examining a small, random subset of the columns of $\bm{A}$, 
which ensures the scalability to high dimensions $n \geq 10^9$.
In the case of sparse columns, the algorithm can even be applied for $n$ so large that the solution cannot be stored as a dense vector on a computer.
The algorithm automatically discovers which entries of the solution vector are significant, leading to a high-accuracy sparse approximation.
We will offer a full mathematical analysis and demonstrate the applicability to large-scale PageRank problems.

The classical Richardson iteration is presented as \cref{alg:ji_classical}.
The method can be applied to any linear system $\bm{A} \bm{x} = \bm{b}$, where $\bm{A}$ and $\bm{b}$ have been scaled so the eigenvalues $\lambda_i(\bm{A})$ all satisfy $|\lambda_i(\bm{A}) - 1| < 1$ (for more discussion of scaling, see \cref{sec:kaczmarz}).
Richardson iteration is based on rewriting $\bm{A} \bm{x} = \bm{b}$ using the fixed-point formula
\begin{equation*}
    \bm{x} = \bm{G} \bm{x} + \bm{b}, \quad \text{where} \quad \bm{G} = \mathbf{I} - \bm{A}.
\end{equation*}
Motivated by this formula, Richardson iteration generates a sequence of approximations:
\begin{equation}
\label{eq:original}
    \begin{cases}
    \bm{x}_0 = \bm{0}, \\
    \bm{x}_s = \bm{G} \bm{x}_{s-1} + \bm{b}.
    \end{cases}
\end{equation}
The iterates $\bm{x}_0, \bm{x}_1, \ldots$ can be interpreted as an application of Horner's rule \cite{horner1819new} for calculating the Neumann series $\bm{x}_\star = \sum_{s=0}^\infty \bm{G}^s \bm{b}$,
and they
converge to the solution vector $\bm{x}_\star$ at an exponential rate specified by $\bm{x}_s = \bm{x}_\star - \bm{G}^s \bm{x}_\star$.
Yet each step of Richardson iteration can be computationally intensive, as it requires a full pass through the entries of $\bm{G}$.

\begin{algorithm}[t]
\caption{Classical Richardson iteration for solving $\bm{A} \bm{x} = \bm{b}$ \cite{richardson1911approximation}} \label{alg:ji_classical}
\begin{algorithmic}[1]
\Require Vector $\bm{b} \in \mathbb{C}^n$; matrix $\bm{A} \in \mathbb{C}^{n \times n}$; iteration count $t$
\Ensure Approximate solution $\bm{x}_t$ to $\bm{A} \bm{x} = \bm{b}$
\State $\bm{x}_0 = \bm{0}$.
\For{$s = 1, 2, \ldots, t$}
\State $\bm{x}_s = (\mathbf{I} - \bm{A}) \bm{x}_{s-1} + \bm{b}$
\EndFor
\State Return $\bm{x}_t$
\end{algorithmic}
\end{algorithm}

Our new approach is called ``randomly sparsified Richardson iteration'' (RSRI).
RSRI is similar to the classical Richardson iteration, except that we
replace the deterministic iteration \cref{eq:original} with the randomized iteration
\begin{equation*}
    \begin{cases}
    \bm{x}_0 = \bm{0}, \\
    \bm{x}_s = \bm{G} \bm{\phi}_s(\bm{x}_{s-1}) + \bm{b},
    \end{cases}
\end{equation*}
where $\bm{\phi}_s$ is a random operator that inputs $\bm{x}_{s-1}$ and outputs a sparse random vector $\bm{\phi}_s(\bm{x}_{s-1})$ with no more than $m$ nonzero entries.
Due to sparsity, it is cheap to evaluate $\bm{G} \bm{\phi}_s(\bm{x}_{s-1})$.
Instead of  a full pass through the matrix at each iteration, the algorithm only requires a multiplication involving a random subset of $m$ columns, where $m$ is a tunable parameter.
The random sparsification introduces errors, which are reduced by averaging over successive iterates
$\bm{x}_{t_{\rm b}}, \bm{x}_{t_{\rm b} + 1}, \ldots, \bm{x}_{t - 1}$. 
See the pseudocode in \cref{alg:RSRI_general}.

We will prove that the RSRI solution $\overline{\bm{x}}_t$ converges as $t \rightarrow \infty$
if $\bm{G}$ is a strict $1$-norm contraction:
\begin{equation*}
    \lVert \bm{G} \rVert_1 =
    \max_{1 \leq j \leq n} \sum\nolimits_{i=1}^n |\bm{G}(i,j)| < 1.
\end{equation*}
Under the contractivity assumption, \cref{thm:main} establishes a quantitative convergence rate.
We will also prove that RSRI converges with weaker requirements on $\bm{G}$, which are defined in \cref{thm:extended}.

\begin{algorithm}[t]
\caption{Randomly sparsified Richardson iteration for solving $\bm{A} \bm{x} = \bm{b}$} \label{alg:RSRI_general}
\begin{algorithmic}[1]
\Require Vector $\bm{b} \in \mathbb{C}^n$; program for evaluating columns of $\bm{A} \in \mathbb{C}^{n \times n}$; parameters $m$, $t_{\rm b}$; iteration count $t$
\Ensure Approximate solution $\overline{\bm{x}}_t$ to $\bm{A} \bm{x} = \bm{b}$
\State $\bm{x}_0 = \bm{0}$
\For{$s = 1, 2, \ldots, t - 1$}
\State $\bm{\phi}_s(\bm{x}_{s-1}) = \texttt{sparsify}(\bm{x}_{s-1}, m)$
\Comment{Sparsify using \cref{alg:optimal}}
\State $\bm{x}_s = (\mathbf{I} - \bm{A}) \bm{\phi}_s(\bm{x}_{s-1}) + \bm{b}$
\EndFor
\State Return $\overline{\bm{x}}_t =  \frac{1}{t - t_{\rm b}} \sum\nolimits_{s = t_{\rm b}}^{t - 1} \bm{x}_s$
\end{algorithmic}
\end{algorithm}

RSRI combines Richardson iteration with random sparsification,
but the approach is \emph{faster} than Richardson iteration for high-dimensional problems and \emph{more accurate} than direct Monte Carlo sampling.
\begin{itemize}
    \item \emph{Speed: } Richardson iteration requires a complete pass through the matrix at each iteration.
    In contrast, RSRI only requires reading $m$ columns of $\bm{A}$ at every iteration, where $m$ is a tunable sparsity level that can be quite small (say, $m \leq n / 10^3$).
    \item \emph{Accuracy: } Classical Monte Carlo strategies for solving linear systems \cite{forsythe1950matrix,wasow1952note} give error bars of size $\sim m^{-1/2}$ where $m$ is the number of samples. 
    Our tunable sparsity parameter $m$ plays a similar role to a number of samples.  
    As we increase $m$, we will highlight settings in which RSRI converges at a polynomial rate $\sim m^{-p}$ for $p > 1/2$ or at an exponential rate $\sim {\rm e}^{-c m}$ for $c > 0$.
\end{itemize}

The core component of RSRI is the random sparsification operator $\bm{\phi}_s$, which inputs a vector $\bm{x}_{s-1} \in \mathbb{C}^n$
and outputs a sparse random vector $\bm{\phi}_s(\bm{x}_{s-1}) \in \mathbb{C}^n$.
We will optimize $\bm{\phi}_s$ in \cref{sec:design}, leading to the high-performing sparsification operator described in \cref{alg:pivotal}.
With some probability $p_i$, the operator replaces the $i$th entry $\bm{x}_{s-1}(i)$ with a higher-magnitude entry $\bm{x}_{s-1}(i) / p_i$;
with the remaining probability $1 - p_i$, the operator sets the $i$th entry to zero.
The probabilities $p_i$ increase proportionally to the magnitude $|\bm{x}_{s-1}(i)|$, reaching $p_i = 1$ for the largest-magnitude entries.
Therefore, these large-magnitude entries are \emph{preserved exactly}.
The combination of randomized rounding and exact preservation leads to an unbiased approximation of the input vector.
Moreover, the input and output vectors are close if the input vector has rapidly decaying entries
(\cref{thm:advantages}).

The per-iteration runtime of RSRI is just $\mathcal{O}(m n)$ operations when $\bm{A}$ is a dense matrix.
The per-iteration runtime of RSRI is even lower --- just $\mathcal{O}(mq)$ operations per iteration --- when $\bm{A}$ and $\bm{b}$ are sparse with no more than $q$ nonzero entries per column.
In the sparse case, the runtime and memory costs are independent of dimension. 
Additionally, if the goal is to compute inner products with the exact solution, the memory cost can be reduced from $\mathcal{O}(tmq)$ to $\mathcal{O}(mq)$ by averaging each inner product over the iterates instead of storing the RSRI solution $\overline{\bm{x}}_t$.

Random sparsification is an approach of growing importance in numerical linear algebra
\cite{forsythe1950matrix,strohmer2008randomized,leventhal2010randomized,lim2017fast}.
One example is stochastic gradient descent, which we contrast with RSRI in \cref{sec:kaczmarz}.
As another example, random sparsification has been applied to eigenvalue problems
in quantum chemistry with matrices as large as $10^{108} \times 10^{108}$, as discussed in \cref{sec:fri}.
RSRI provides a new instantiation of the random sparsification approach for linear systems, and the present work gives mathematical and empirical demonstrations of RSRI's effectiveness.



\subsection{Plan for paper}
The rest of this paper is organized as follows.
\Cref{sec:error_bounds} presents our main error bound for RSRI,
\cref{sec:pagerank} applies RSRI to PageRank problems,
\cref{sec:related} discusses algorithms related to RSRI,
\cref{sec:design} analyzes random sparsification,
and \cref{sec:RSRI_error} proves our main error bound for RSRI.

\subsection{Notation}

We use the shorthand $\lfloor a \rfloor = \max\{z \in \mathbb{Z}:\, z \leq a\}$ and
$a \vee b = \max\{a, b\}$ for $a, b \in \mathbb{R}$.
The complex conjugate of $z \in \mathbb{C}$ is $\overline{z}$.
We write vectors $\bm{v} \in \mathbb{C}^n$ and matrices $\bm{M} \in \mathbb{C}^{n \times n}$ in bold, and we write their elements as $\bm{v}(i)$ or $\bm{M}(i,j)$.
The conjugate transposes are
$\bm{v}^\ast$ and $\bm{M}^\ast$, while $|\bm{v}|$ and $|\bm{M}|$ denote the entry-wise absolute values.
For any vector $\bm{x} \in \mathbb{C}^n$, the decreasing rearrangement 
$\bm{x}^{\downarrow} \in \mathbb{C}^n$
is a vector with the same elements as $\bm{x}$ but placed in weakly decreasing order:
\begin{equation*}
    |\bm{x}^{\downarrow}(1)| \geq |\bm{x}^{\downarrow}(2)| \geq \cdots \geq |\bm{x}^{\downarrow}(n)|.
\end{equation*}
The decreasing rearrangement may not be unique, so we employ the notation only in
contexts where it leads to an unambiguous statement.
The vector $1$-norm, Euclidean norm, and $\infty$-norm are $\lVert \bm{v} \rVert_1 = \sum\nolimits_{i=1}^n |\bm{v}(i)|$, $\lVert \bm{v} \rVert = (\sum\nolimits_{i=1}^n |\bm{v}(i)|^2)^{1/2}$ and $\lVert \bm{v} \rVert_{\infty} = \max_{1 \leq i \leq n} |\bm{v}(i)|$.
The number of nonzero entries is $\lVert \bm{v} \rVert_0 = \#\{1 \leq i \leq n:\,\bm{v}(i) \neq 0\}$.
The matrix $1$-norm is $\lVert \bm{M} \rVert_1 = \max_{\lVert \bm{v} \rVert_1 = 1} \lVert \bm{M} \bm{v} \rVert_1$.

\section{Main error bound for RSRI} \label{sec:error_bounds}

Our main result is the following detailed error bound for RSRI whose proof appears in \cref{sec:RSRI_error}.

\begin{theorem}[Main error bound] \label{thm:main}
Suppose RSRI with sparsity level $m$ is applied to an $n \times n$ linear system $\bm{A} \bm{x} = \bm{b}$ for which $\bm{G} = \mathbf{I} - \bm{A}$ is a strict $1$-norm contraction:
\begin{equation*}
    \lVert \bm{G} \rVert_1 =
    \max_{1 \leq j \leq n} \sum\nolimits_{1 \leq i \leq n} |\bm{G}(i,j)| < 1.
\end{equation*}
RSRI returns a solution $\overline{\bm{x}}_t$ satisfying the bias-variance formula
\begin{equation}
\label{eq:def}
    \mathbb{E} \bigl\lVert \bm{A} \overline{\bm{x}}_t - \bm{b} \bigr\rVert^2
    = \underbrace{\bigl\lVert \bm{A} \,\mathbb{E}[ \overline{\bm{x}}_t] - \bm{b} \bigr\rVert^2}_{\textup{bias}^2}
    + \underbrace{\mathbb{E} \bigl\lVert \bm{A} \overline{\bm{x}}_t - \bm{A} \,\mathbb{E}[\overline{\bm{x}}_t] \bigr\rVert^2}_{\textup{variance}}.
\end{equation}
Here the expectation averages over the random set of entries rounded to zero at each sparsification step.
The square bias is bounded by
\begin{equation*}
    \textup{bias}^2
    \leq \biggl(\frac{2 \lVert \bm{G}^{t_{\rm b}} \bm{x}_\star \bigr\rVert_1}{t - t_{\rm b}}\biggr)^2,
\end{equation*}
where $\bm{x}_\star$ is the exact solution.
The variance is bounded by
\begin{equation*}
    \textup{variance}
    \leq \frac{8 t}{(t - t_{\rm b})^2} \cdot \frac{1}{m} \biggl( 
    \frac{\lVert \bm{b} \rVert_1}{1 - \lVert \bm{G} \rVert_1}\biggr)^2.
\end{equation*}
Additionally, if $m \geq m_{\bm{G}} = 1 / (1 - \lVert \bm{G} \rVert_1^2)$, the variance is bounded by
\begin{equation}
\label{eq:var_bound}
    \textup{variance} 
    \leq \frac{8 t}{(t - t_{\rm b})^2} \cdot
    \min_{i \leq m - m_{\bm{G}}} \frac{1}{m - m_{\bm{G}} - i} 
    \biggl(\sum\nolimits_{j = i + 1}^n \tilde{\bm{x}}^{\downarrow}(j)\biggr)^2.
\end{equation}
Here $\tilde{\bm{x}} \in \mathbb{R}^n$ is the solution to the regularized linear system
$\tilde{\bm{A}} \tilde{\bm{x}} = \tilde{\bm{b}}$, where $\tilde{\bm{A}} = \mathbf{I} - |\bm{G}|$, $\tilde{\bm{b}} = |\bm{b}|$. As usual, $\tilde{\bm{x}}^{\downarrow}$ is the decreasing rearrangement of $\tilde{\bm{x}}$.
\end{theorem}

In \cref{thm:main}, we bound the mean square error of RSRI as the sum of a square bias term and a variance term.
The square bias in RSRI is the square bias that occurs in the deterministic Richardson iteration, after averaging over the iterates $\bm{x}_{t_{\rm b}}, \ldots, \bm{x}_{t-1}$ (see \cref{prop:bias_bound}).
The square bias decays exponentially fast as we increase the burn-in time $t_{\rm b}$ and it is often small in practice.
For example, in our experiments in \cref{sec:pagerank}, we obtain a negligibly small bias by setting $t_{\rm b} = t/2$.

The variance term in \cref{thm:main} decays at a rate $\mathcal{O}(1/m)$ or faster, depending on the decay of the entries in the regularized solution vector $\tilde{\bm{x}}$.
Note that the entries of the regularized solution $\tilde{\bm{x}}$ lie above the entries of the true solution $\bm{x}_\star$ due to the element-wise inequality
\begin{equation*}
    |\bm{x}_\star| 
    = \Bigl| \sum\nolimits_{s=0}^{\infty} \bm{G}^s \bm{b} \Bigr|
    \leq \sum\nolimits_{s=0}^{\infty} |\bm{G}|^s |\bm{b}| = \tilde{\bm{x}}.
\end{equation*}
The vectors $\tilde{\bm{x}}$ and $\bm{x}_\star$ are identical if $\bm{G}$ and $\bm{b}$ are nonnegative-valued, which is the case for PageRank problems (see \cref{sec:pagerank}).

\Cref{thm:main} leads to the main message of this work.
There is a class of problems in which the entries of $\tilde{\bm{x}}$ are decreasing quickly,
and RSRI converges at a fast polynomial or exponential rate.
We establish the following corollary of \cref{thm:main}.

\begin{corollary}[Fast polynomial or exponential convergence] \label{cor:explicit}
Instate the notation of \cref{thm:main}.
If the sparsity level satisfies $m \geq m_{\bm{G}} := 1 / (1 - \lVert \bm{G} \rVert_1)$ and
\begin{equation*}
    \sum\nolimits_{j = i}^n \tilde{\bm{x}}^{\downarrow}(j) \leq i^{-p}
    \quad \text{for} \quad 1 \leq i \leq n,
\end{equation*}
then the RSRI variance \cref{eq:def} is bounded by
\begin{equation*}
    \textup{variance}
    \leq 16 {\rm e} \frac{(p + \frac{1}{2}) t}{(t - t_{\rm b})^2}
    \cdot (m-m_{\bm{G}})^{-2(p+\frac{1}{2})}.
\end{equation*}
If the sparsity level satisfies $m \geq m_{\bm{G}} + \frac{1}{2 c}$ and
\begin{equation*}
    \sum\nolimits_{j = i}^n \tilde{\bm{x}}^{\downarrow}(j) \leq {\rm e}^{-c i}
    \quad \text{for} \quad 1 \leq i \leq n,
\end{equation*}
then the RSRI variance \cref{eq:def} is bounded by
\begin{equation*}
    \textup{variance}
    \leq 16 {\rm e} \frac{c t}{(t - t_{\rm b})^2}
    \cdot {\rm e}^{-2 c (m-m_{\bm{G}})}.
\end{equation*}
\end{corollary}
\begin{proof}
For the first bound, we take $i = \bigl\lfloor \frac{2p}{2p + 1} (m - m_{\bm{G}}) \bigr\rfloor$ and evaluate \cref{eq:var_bound}.
For the second bound, we take $i = \bigl\lfloor m - m_{\bm{G}} - \frac{1}{2 c} \bigr\rfloor$ and evaluate \cref{eq:var_bound} again.
\end{proof}

For specific problems in which the regularized solution vector $\tilde{\bm{x}}$ is quickly decaying, \cref{cor:explicit} establishes that RSRI is more efficient than a pure Monte Carlo method.
If the tail $\sum\nolimits_{j = i}^n \tilde{\bm{x}}^{\downarrow}(j)$ decays polynomially with rate $i^{-p}$,
then RSRI converges at the rate $\mathcal{O}(m^{-p-1/2})$.
Additionally, if the tail $\sum\nolimits_{j = i}^n \tilde{\bm{x}}^{\downarrow}(j)$ decays exponentially with rate ${\rm e}^{-\Delta i}$, then RSRI converges at the exponential rate $\mathcal{O}({\rm e}^{-\Delta m})$.
In the next section, we will show examples of PageRank problems where the solution vector exhibits polynomial tail decay.

\section{The PageRank problem} \label{sec:pagerank}

Consider a network of $n$ websites.
A restless web surfer chooses an initial website at random, according to a probability vector $\bm{s} \in [0,1]^n$.
At any time, there is a fixed probability $\alpha \in (0, 1)$ that the web surfer follows a random hyperlink. In this case, the probability of transitioning from website $j$ to website $i$ is denoted as $\bm{P}(i,j)$.
With the remaining probability $1 - \alpha$, the web surfer abandons the chain of websites and chooses a fresh website according to the probability vector $\bm{s}$.
The PageRank problem asks: what is the long-run distribution of websites visited by the web surfer?

The PageRank problem can be solved using a linear system of equations.
We let $\bm{x} \in \mathbb{R}^n$ denote the long-run distribution of visited websites and let $\bm{P} \in \mathbb{R}^{n \times n}$ denote the column-stochastic matrix of transition probabilities. 
Then $\bm{x}$ satisfies
\begin{equation}
\label{eq:richardson}
    \bm{x} = \alpha \bm{P} \bm{x} + (1 - \alpha) \bm{s}.
\end{equation}
By setting $\bm{A} = \mathbf{I} - \alpha \bm{P}$ and $\bm{b} = (1 - \alpha) \bm{s}$, we can rewrite this linear system in the standard form
\begin{equation*}
    \bm{A} \bm{x} = \bm{b},
\end{equation*}
where the matrix $\bm{A}$ is typically sparse and high-dimensional.

Since the late 1990s, Google has applied PageRank with $\alpha = .85$ to help determine which websites show up first in their search results.
Originally, Google solved the PageRank problem using the Richardson fixed-point iteration based on \cref{eq:richardson}.
For a network of $N = 2.4 \times 10^7$ websites with $3.2 \times 10^8$ hyperlinks, the developers of PageRank report that Richardson iteration converges to the target accuracy in 52 iterations, which made the algorithm practical for ranking websites in 1998 \cite{page1999pagerank}.
However, the 2025 internet has a greater number of websites and hyperlinks, leading to a more challenging PageRank problem.

Beyond the internet, PageRank problems arise in chemistry, biology, and data science, among other areas
\cite{gleich2015pagerank}.
We focus especially on \emph{personalized PageRank problems}, in which the probability vector $\bm{s}$ has just one or a few nonzero entries.
The personalized PageRank problem identifies which vertices are important in a local region associated with the nonzero entries of $\bm{s}$.
For example, personalized PageRank problems arise when recommending items to particular users on Netflix or Amazon \cite{gori2007itemrank}, when disambiguating the meanings of words in a sentence \cite{agirre2009personalizing}, or when constructing personalized reading lists \cite{Wis06}.

The rest of this section is organized as follows.
\Cref{sec:pagerank_history} gives a history of Monte Carlo algorithms for large-scale personalized PageRank problems, \cref{sec:faster} presents a mathematical analysis showing that RSRI is effective for the PageRank problems, and \cref{sec:empirical} presents empirical tests.
Alternative PageRank algorithms with different design principles are discussed in \cref{sec:coordinate_pagerank}.

\subsection{History of Monte Carlo PageRank algorithms} \label{sec:pagerank_history}

Personalized PageRank problems depend most strongly on the vertices associated with the nonzero entries of $\bm{s}$ and nearby regions,
so a \emph{local} search over the vertices can in principle produce a high-quality approximation.
Based on this insight,
several researchers have proposed algorithms to identify the most important vertices and sparsely approximate the PageRank solution vector.
The Monte Carlo algorithms of \cite{fogaras2005towards,avrachenkov2007monte,borgs2014multiscale} simulate a web surfer that moves from vertex to vertex, and the visited vertices determine the sparsity pattern.
In contrast, the deterministic algorithms of \cite{jeh2003scaling,berkhin2006bookmark,gleich2006approximating,sarlos2006to,andersen2007local} identify a set of significant vertices by progressively adding new  vertices for which the residual is large.

Monte Carlo algorithms for the PageRank problem converge at the typical $\sim m^{-1/2}$ error rate, where $m$ is the number of samples.
For example, the Monte Carlo scheme of Fogaras et al.~\cite{fogaras2005towards} simulates $m$ restless web surfers until the first time they become bored and abandon the search.
The last websites visited by the surfers are recorded as $Z_1, Z_2, \ldots, Z_m \in \{1, \ldots, n\}$.
Since each $Z_i$ is an independent sample from the PageRank distribution, the PageRank vector $\bm{x}_\star \in \mathbb{R}^n$ is approximated as $\hat{\bm{x}} = \frac{1}{m} \sum\nolimits_{i=1}^m \bm{e}_{Z_i}$ where $\bm{e}_j$ denotes the $j$th basis vector.
The approximation satisfies the following sharp variance bound:
\begin{equation}
\label{eq:calculation}
\begin{aligned}
\mathbb{E} \lVert \hat{\bm{x}} - \bm{x}_\star \rVert^2 
&= \sum\nolimits_{i=1}^n \mathbb{E} | \hat{\bm{x}}(i) - \bm{x}_\star(i) |^2
= \frac{1}{m} \sum\nolimits_{i=1}^n \bm{x}_\star(i) (1 - \bm{x}_\star(i)) \\
&\leq \frac{1}{m} \sum\nolimits_{i=1}^n \bm{x}_\star(i)
= \frac{1}{m}.
\end{aligned}
\end{equation}
Here, we use the fact that each $\hat{\bm{x}}(i)$ is a rescaled binomial random variable with parameters $m$ and $\bm{x}_\star(i)$, so the variance is $\frac{1}{m} \bm{x}_\star(i) (1 - \bm{x}_\star(i))$.
The error bound \cref{eq:calculation} is completely independent of the dimension, but it signals a slow $m^{-1/2}$ convergence in the root-mean-square error as we increase the number of samples $m$.
Alternative Monte Carlo algorithms such as \cite[Alg.~3]{avrachenkov2007monte} and \cite[Alg.~3]{borgs2014multiscale} improve the runtime and variance by a constant factor, but they do not change the fundamental $m^{-1/2}$ convergence rate.

\subsection{Faster PageRank by RSRI} \label{sec:faster}

When we apply RSRI to solve the personalized PageRank problem and we use the minimal sparsity setting $m = 1$, it is equivalent to a standard Monte Carlo algorithm \cite[Alg.~3]{avrachenkov2007monte}.
However, as we raise $m$, RSRI exhibits more complicated behavior and it satisfies the following error bound, with the proof appearing in \cref{sec:RSRI_error}.
\begin{proposition}[PageRank error bound] \label{prop:pagerank}
If randomly sparsified Richardson iteration is applied to the PageRank problem $\bm{x} = \alpha \bm{P} \bm{x} + (1 - \alpha) \bm{s}$ with parameters $m \geq m_{\alpha} := 1/(1 - \alpha^2)$ and $t \geq 2 t_{\rm b}$, RSRI returns a solution vector $\overline{\bm{x}}_t$ satisfying
\begin{equation*}
    \mathbb{E} \bigl\lVert \overline{\bm{x}}_t - \bm{x}_\star \bigr\rVert^2
    \leq \biggl[\frac{4 \alpha^{t_{\rm b}}}{(1 - \alpha) t}\biggr]^2
    + \frac{16}{(1 - \alpha)^2 t}
    \cdot \min_{0 \leq i \leq m - m_{\alpha}} \frac{1}{m - m_{\alpha} - i} 
    \biggl[\sum\nolimits_{j = i+1}^n \bm{x}_\star^{\downarrow}(j)\biggr]^2,
\end{equation*}
where $\bm{x}_\star^{\downarrow}$ is the decreasing rearrangement of the solution vector $\bm{x}_\star$.
\end{proposition}

\Cref{prop:pagerank} bounds the mean square PageRank error in the Euclidean norm.
The variance term decays at a rate at least $\mathcal{O}(m^{-1})$, and it decays more rapidly than $\mathcal{O}(m^{-1})$ when the entries of $\bm{x}$ are decaying rapidly.
Moreover, we observe that the entries of $\bm{x}$ must decay rapidly if $\bm{P}$ is sparse, which is typically the case in PageRank problems.
We present the following sparsity-based estimate.
\begin{lemma}[Decay of entries in the PageRank vector] \label{lem:decay_rate}
Consider the personalized PageRank problem 
$\bm{x} = \alpha \bm{P} \bm{x} + (1 - \alpha) \bm{e}_i$, 
and assume each column of $\bm{P}$ has at most $q$ nonzero entries, where $q \geq 2$.
Then the entries of the solution $\bm{x}_\star$ satisfy
\begin{equation*}
    \sum\nolimits_{j = i}^n \bm{x}_\star^{\downarrow}(j)
    \leq \alpha^{-1} i^{-\log_q(1/\alpha)}.
\end{equation*}
\end{lemma}
\begin{proof}[Proof of \cref{lem:decay_rate}]
Let $\textup{dist}(i, j)$ denote the length of the shortest path 
\begin{equation*}
    i = k_0 \rightarrow k_1 \rightarrow \cdots \rightarrow k_{s-1} \rightarrow k_s = j,
\end{equation*}
which has positive probability of occurring, i.e., $\bm{P}(k_r,k_{r-1}) > 0$ for each $r = 1, \ldots, s$.
By the sparsity condition, there can be at most $1 + q + q^2 + \cdots + q^{s-1} < q^s$ paths of length $s - 1$ or shorter.
Now let $\sigma_1, \ldots, \sigma_n$ be a permutation of the indices $1, \ldots, n$ so that $\sigma_1 = i$ and $\textup{dist}(i, \sigma_j) \leq \textup{dist}(i, \sigma_k)$ whenever $j \leq k$.
For any index $j \geq q^s$, we must have $\textup{dist}(i, \sigma_j) \geq s$.

Now we use the representation $\bm{x}_\star = (1 - \alpha) \sum\nolimits_{s=0}^{\infty} \alpha^s \bm{P}^s \bm{e}_i$ and the fact that $\bm{P}^s$ is column-stochastic to calculate
\begin{align*}
    \sum\nolimits_{j \geq m} \bm{x}_\star(\sigma_j)
    &= (1 - \alpha) \sum\nolimits_{j \geq m}
    \sum\nolimits_{s=0}^{\infty} \alpha^s \bm{P}^s (\sigma_j, i) \\
    &= (1 - \alpha) \sum\nolimits_{j \geq m}
    \sum\nolimits_{s=\lfloor \log_q(m) \rfloor}^{\infty} \alpha^s \bm{P}^s(\sigma_j, i) \\
    & \leq 
    (1 - \alpha) \sum\nolimits_{s=\lfloor \log_q(m) \rfloor}^{\infty} \alpha^s
    = \alpha^{\lfloor \log_q(m) \rfloor}
    \leq
    \alpha^{-1} m^{-\log_q(1/\alpha)},
\end{align*}
which establishes the result.
\end{proof}

By combining, \cref{prop:pagerank,lem:decay_rate}, we guarantee that RSRI converges at a $\mathcal{O}(m^{-1/2 - \log_q(1/\alpha)})$ rate as we increase the sparsity parameter $m$.
This rate is faster than the $\sim m^{-1/2}$ Monte Carlo error scaling, and it theoretically separates RSRI from pure Monte Carlo methods.
For many PageRank problems, we expect even faster convergence than \cref{lem:decay_rate} would suggest.
For example, the PageRank solution decays especially rapidly when there are ``hub'' vertices which have many incoming edges \cite{jeh2003scaling}.

\subsection{Empirical tests} \label{sec:empirical}

To test the empirical performance of RSRI, we apply the algorithm to three personalized PageRank problems:
\begin{itemize}
    \item \textbf{Amazon electronics} \cite{ni2019justifying,anand2019amazon}. 
    We rank $n = 3.6 \times 10^5$ Amazon electronics products that received five-star reviews and were available in 2019.
    The transitions probabilities are determined by linking from an electronics product to a different random product receiving a 5-star review from the same reviewer.
    \item \textbf{Notre Dame websites} \cite{albert1999diameter,leskovec1999notre}. 
    We rank $n = 3.2 \times 10^5$ websites within the 1999 University of Notre Dame web domain.
    The transition probabilities are determined by selecting a random outgoing hyperlink.
    \item \textbf{Airports} \cite{opsahl2011why}.
    We rank $n = 2.9 \times 10^3$ global airports.
    The transition probabilities are determined by selecting a random outgoing flight from the documented flight routes in 2010.
\end{itemize}

In all three problems, we compute the personalized PageRank vector for a randomly chosen vertex $i \in \{1, \ldots, n\}$.
If there is a dangling vertex $j$ which lacks outgoing edges, 
we update the $j$th column of the transition matrix $\bm{P}$ to be the basis vector $\bm{e}_i$.
This is equivalent to solving the unnormalized problem $\bm{x} = \alpha \bm{P} \bm{x} + (1 - \alpha) \bm{e}_i$ and rescaling the solution $\bm{x}$ to sum to one \cite[Thm.~2.5]{gleich2015pagerank}.
We apply RSRI with parameters $\alpha = .85$, $t = 1000$, $t_{\rm b} = t/2$
and present the root mean square error $(\mathbb{E} \lVert \hat{\bm{x}}_t - \bm{x}_\star \rVert^2)^{1/2}$
in \cref{fig:error} (right).
\begin{figure}[t]
    \centering
    \includegraphics[scale=.4]{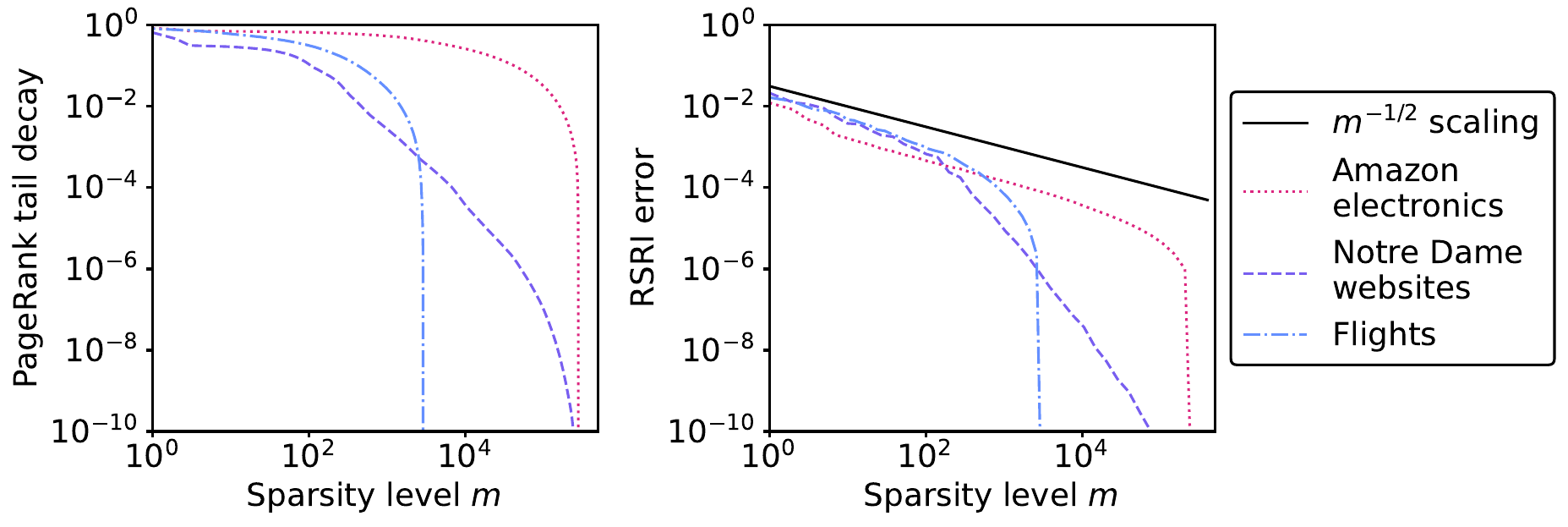}
    \caption{\textbf{(RSRI error scaling.)}
    Left panel shows tail decay $\sum\nolimits_{i = m}^n \bm{x}^{\downarrow}(i)$ of the sorted PageRank solution $\bm{x}^{\downarrow}$ for three personalized PageRank problems documented in \cref{sec:empirical}.
    Right panel shows RSRI error $(\mathbb{E} \lVert \hat{\bm{x}} - \bm{x}_\star \rVert^2)^{1/2}$ with sparsity level $m$.
    The right panel is essentially the left panel multiplied by a factor of $m^{-1/2}$}
    \label{fig:error}
\end{figure}
The expectation is evaluated empirically over ten independent trials.
See \url{https://github.com/rjwebber/rsri} for code for all experiments in this paper.

The results verify that RSRI provides a high-accuracy solution for all three PageRank problems, reaching error levels of $10^{-6}$--$10^{-3}$ even when $m \leq n / 10^2$.
Moreover, the error decays at a faster-than-$m^{-1/2}$ rate as we increase $m$.
The precise rate of convergence depends on the tail decay $\sum\nolimits_{i = m}^n \bm{x}_\star^{\downarrow}(i)$.
The RSRI error (\cref{fig:error}, right) is approximately the PageRank tail decay (\cref{fig:error}, left) multiplied by a $\mathcal{O}(m^{-1/2})$ prefactor term.
In accordance with \cref{prop:pagerank},
RSRI converges fastest when the tail decays fastest, which occurs in the Notre Dame websites problem.

\section{Related algorithms} \label{sec:related}

Here we review the algorithms most closely related to RSRI.

\subsection{The Monte Carlo method for linear systems} \label{sec:monte_carlo}
In the late 1940s, Ulam and von Neumann introduced a Monte Carlo strategy for solving linear systems $\bm{A} \bm{x} = \bm{b}$ \cite{forsythe1950matrix}.
To motivate this strategy, let us write the solution vector $\bm{x}_\star$ as
\begin{equation*}
    \bm{x}_\star = \sum\nolimits_{s=0}^{\infty} \bm{G}^s \bm{b}, \quad \text{where} \quad \bm{G} = \mathbf{I} - \bm{A},
\end{equation*}
and let us assume that $\bm{b}$ and $\bm{G}$ have nonnegative-valued entries with $\sum\nolimits_{i=1}^n \bm{b}(i) = 1$ and $\sum\nolimits_{i=1}^n \bm{G}(i,j) < 1$ for $1 \leq j \leq n$.
Ulam and von Neumann interpreted each term $\bm{G}^s \bm{b}$ using a finite-state Markov chain $(X_s)_{s \geq 0}$ with transition probabilities
\begin{equation}
\label{eq:stochastic1}
    \mathbb{P}\{X_0 = i\} = \bm{b}(i), \qquad \mathbb{P}\{X_s = i | X_{s-1} = j\} = \bm{G}(i,j).
\end{equation}
The transition probabilities $(\bm{G}(i,j))_{1 \leq i \leq n}$ sum to less than one, so there is a positive probability
\begin{equation}
\label{eq:stochastic2}
    p_j := 1 - \sum\nolimits_{i=1}^n \bm{G}(i,j).
\end{equation}
that the Markov chain occupying state $j$ at time $s$ is \emph{killed}.
In this case, define the killing time $\tau$ to be $\tau = s$.

As a practical algorithm, Ulam and von Neumann suggested simulating the Markov chain with killing on a computer and estimating the vector $\bm{x}_\star$ using
\begin{equation}
\label{eq:unbiased}
    \hat{\bm{x}} = \frac{1}{p_{X_{\tau}}} \bm{e}_{X_{\tau}},
\end{equation}
where $p_j$ is the $j$th killing probability \cref{eq:stochastic2}, $\bm{e}_j$ is the $j$th basis vector, and $X_{\tau}$ is the state of the Markov chain when the random killing occurs.
This stochastic estimator is unbiased but has a high variance, so
Ulam and von Neumman proposed averaging over many independent estimators $\hat{\bm{x}}$ to bring down the variance.
They also introduced a more complicated Monte Carlo procedure for solving systems in which $\bm{b}$ and $\bm{G}$ can have negative-valued entries \cite{forsythe1950matrix}, and further improvements were introduced by Wasow \cite{wasow1952note} soon after.

Since the 1950s, computational scientists have applied the Monte Carlo method to solve large-scale linear systems arising from numerical discretizations of PDEs \cite{sadeh1974monte,wang2008monte,evans2014monte,benzi2017analysis} and integral equations \cite{lai2009adaptive}.
The method has been analyzed by numerical analysts \cite{forsythe1950matrix,wasow1952note,edmundson1953monte,curtiss1953monte,bauer1958monte,dimov1991minimization,okten2005solving,ji2013convergence} and theoretical computer scientists \cite{ozdaglar2020asynchronous,shyamkumar2016sublinear,andoni2018solving}.
The Monte Carlo method was even rediscovered in the PageRank community, leading to the Monte Carlo PageRank algorithms of Fogaras et al. (2005) \cite{fogaras2005towards} and Avrachenkov et al. (2007) \cite{avrachenkov2007monte} as discussed in \cref{sec:pagerank_history}.
Despite all this research, however, the method is fundamentally limited by a $ m^{-1/2}$ convergence rate.
Modifications to the basic procedure do not fundamentally change the convergence rate, or else they change the character of the algorithm by reading $\mathcal{O}(n)$ columns per iteration \cite{halton1994sequential}.

Randomly sparsified Richardson iteration is based on random sampling and is a Monte Carlo method in the case $m = 1$.
However, for $m > 1$, RSRI improves on past Monte Carlo methods.
We are not aware of another
Monte Carlo-based algorithm that achieves faster-than-$m^{-1/2}$ convergence while requiring just $\mathcal{O}(m)$ column evaluations per iteration.

\subsection{Stochastic gradient descent}
\label{sec:kaczmarz}

Stochastic gradient descent (SGD) is a class of methods that speed up traditional gradient descent by subsampling the terms in an expansion of the gradient  \cite{rosenblatt1958perceptron,needell2015stochastic,garrigos2024handbookconvergencetheoremsstochastic}.
Like traditional gradient descent, SGD is designed to minimize a scalar loss function $L: \mathbb{R}^d \rightarrow \mathbb{R}$.
When solving a linear system $\bm{A} \bm{x} = \bm{b}$, there are a couple canonical loss functions that lead to a minimizer $\bm{x} = \bm{A}^{-1} \bm{b}$.
\begin{itemize}
\item[(a)] When $\bm{A}$ is strictly positive definite, one choice of loss function is $L_1(\bm{x}) = \frac{1}{2} \bm{x}^* \bm{A} \bm{x} - \bm{x}^* \bm{b}$.
\item[(b)] When $\bm{A}$ is any invertible matrix, a different choice of loss function is $L_2(\bm{x}) = \frac{1}{2} \lVert \bm{A} \bm{x} - \bm{b} \rVert^2$.
\end{itemize}
Gradient descent methods minimize either of these loss functions by taking steps in the negative gradient direction, $-\nabla L_1(\bm{x})$ or $-\nabla L_2(\bm{x})$.
In contrast, SGD methods subsample just one or a few terms from the gradient expansion
\begin{align*}
    & \nabla L_1(\bm{x}) = \sum\nolimits_{j=1}^n \bm{g}_{1,j}(\bm{x}),
    \quad \text{where} \quad
    \bm{g}_{1,j}(\bm{x}) = \bigl[\bm{A}(j, \cdot) \bm{x} - \bm{b}(j)\bigr] \bm{e}_j, \\
    \text{or} \quad &\nabla L_2(\bm{x}) = \sum\nolimits_{j=1}^n \bm{g}_{2,j}(\bm{x}), \quad \text{where} \quad
    \bm{g}_{2,j}(\bm{x}) = \bigl[\bm{A}(j, \cdot) \bm{x} - \bm{b}(j)\bigr] \bm{A}(j, \cdot)^*.
\end{align*}
For example, uniform mini-batch SGD \cite{gower2019sgd,garrigos2024handbookconvergencetheoremsstochastic} samples an index set $\textsf{S} \subseteq \{1, \ldots, n\}$ uniformly at random without replacement.
Then, it makes an update
\begin{equation*}
    \bm{x}_s = \bm{x}_{s-1} - \frac{\alpha n}{|\textsf{S}|} \sum\nolimits_{j \in \textsf{S}} \bm{g}_{i,j}(\bm{x}_{s-1}),
\end{equation*}
where $\alpha > 0$ is a step size parameter and $|\textsf{S}|$ is the cardinality of $\textsf{S}$.
In contrast, importance sampling SGD \cite{needell2015stochastic,moorman2020randomized} samples $\textsf{S}$ with replacement from a nonuniform probability distribution $(p_i)_{1 \leq i \leq n}$.
Then, it makes an update
\begin{equation*}
    \bm{x}_s = \bm{x}_{s-1} - \frac{\alpha}{|\textsf{S}|} \sum\nolimits_{j \in \textsf{S}} \frac{\bm{g}_{i,j}(\bm{x}_{s-1})}{p_j}.
\end{equation*}
Two types of importance sampling SGD are especially well-known.
Randomized coordinate descent \cite{leventhal2010randomized,qu2016coordinate} uses unequal sampling probabilities $p_j = \bm{A}(j, j) / \operatorname{tr}(\bm{A})$ to optimize the loss function $L_1$, while randomized Kaczmarz \cite{strohmer2008randomized,moorman2020randomized} uses unequal sampling probabilities $p_j = \lVert \bm{A}(j, \cdot) \rVert^2 / \lVert \bm{A} \rVert_{\rm F}^2$ to optimize the loss function $L_2$.

On the surface, SGD is similar to RSRI. 
The expected value of the update in SGD is:
\begin{align*}
    & \mathbb{E} \bigl[\bm{x}_s - \bm{x}_{s-1} | \bm{x}_{s-1}\bigr]
    = -\alpha \nabla L_1(\bm{x}_{s-1}) 
    = \alpha (\bm{b} - \bm{A} \bm{x}_{s-1}) \\
    \text{or} \quad  &
    \mathbb{E} \bigl[\bm{x}_s - \bm{x}_{s-1} | \bm{x}_{s-1}\bigr]
    = -\alpha \nabla L_2(\bm{x}_{s-1})
    = \alpha \bm{A}^* (\bm{b} - \bm{A} \bm{x}_{s-1}).
\end{align*}
RSRI makes exactly the same update in expectation when applied to $\alpha \bm{A} \bm{x} = \alpha \bm{b}$ or $\alpha \bm{A}^* \bm{A} \bm{x} = \alpha \bm{A}^* \bm{b}$.

Although the SGD and RSRI updates point in the same direction in expectation, the step size is chosen differently in these methods.
When RSRI is applied to $\alpha \bm{A} \bm{x} = \alpha \bm{b}$ or $\alpha \bm{A}^* \bm{A} \bm{x} = \alpha \bm{A}^* \bm{b}$, the step size can be any value $\alpha \in (0, 2 / \lVert \bm{A} \rVert)$ or $\alpha \in (0, 2 / \lVert \bm{A} \rVert^2)$ and still the algorithm converges under $\ell_1$ contractivity assumptions (\cref{thm:extended}).
In contrast, the step size needs to be much smaller for SGD methods with a constant batch size to converge.
For example, randomized coordinate descent \cite{leventhal2010randomized,qu2016coordinate} produces iterates that satisfy
\begin{multline}
    \label{eq:rcd_takedown}
        \mathbb{E} \lVert \bm{A}^{1/2} (\bm{x}_s - \bm{x}_\star) \rVert^2 = \mathbb{E} \lVert \bm{A}^{1/2} (\bm{x}_{s-1} - \bm{x}_\star) \rVert^2 \\
        + \biggl[ \frac{\alpha^2 \operatorname{tr}(\bm{A})}{|\textsf{S}|} - 2\alpha \biggr] \mathbb{E} \lVert \bm{A} (\bm{x}_{s-1} - \bm{x}_\star) \rVert^2 + \alpha^2 \frac{|\textsf{S}|-1}{|\textsf{S}|} \mathbb{E} \lVert \bm{A}^{3/2} (\bm{x}_{s-1} - \bm{x}_\star) \rVert^2.
\end{multline}
Meanwhile, randomized Kaczmarz \cite{strohmer2008randomized,moorman2020randomized} produces iterates that satisfy
\begin{multline}
\label{eq:rk_takedown}
    \mathbb{E} \lVert \bm{x}_s - \bm{x}_\star \rVert^2
    = \mathbb{E} \lVert \bm{x}_{s-1} - \bm{x}_\star \rVert^2 \\
    + \biggl[\alpha^2 \frac{\lVert \bm{A} \rVert_{\rm F}^2}{|\textsf{S}|} - 2 \alpha\biggr] \mathbb{E} \lVert \bm{A} (\bm{x}_{s-1} - \bm{x}_\star) \rVert^2
    + \alpha^2 \frac{|\textsf{S}|-1}{|\textsf{S}|} \mathbb{E} \lVert \bm{A}^* \bm{A} (\bm{x}_{s-1} - \bm{x}_\star) \rVert^2.
\end{multline}
See \cref{sec:discussion} for a derivation.
Thus, the SGD iterates diverge when $\alpha > 2 |\textsf{S}| / \operatorname{tr}(\bm{A})$ for randomized coordinate descent or $\alpha > 2 |\textsf{S}| / \lVert \bm{A} \rVert_{\rm F}^2$ for randomized Kaczmarz.
In uniform mini-batch SGD \cite{gower2019sgd,garrigos2024handbookconvergencetheoremsstochastic},
the step size parameter must be taken even smaller to prevent divergence.
Indeed, randomized Kaczmarz and randomized coordinate descent can be interpreted as optimized SGD methods that take maximally large steps while ensuring stability \cite{needell2015stochastic}.

In summary, RSRI uses a dimension-independent step size $\alpha \in (0, 2/\lVert \bm{A} \rVert)$ or $\alpha \in (0, 2/\lVert \bm{A} \rVert^2)$, which can be larger than the SGD step size $\alpha \in (0, 2|\textsf{S}|/ \operatorname{tr}(\bm{A}))$ or $\alpha \in (0, 2|\textsf{S}| / \lVert \bm{A} \rVert^2_{\rm F})$ by a factor as high as $n / |\textsf{S}|$.
The small step size in SGD leads to a slow, dimension-dependent convergence rate.
Indeed, optimizing the convergence rate in \cref{eq:rcd_takedown} or \cref{eq:rk_takedown} leads to bounds that are well-known in the case $|\textsf{S}| = 1$ \cite{strohmer2008randomized,leventhal2010randomized}
and hold if $(|\textsf{S}| - 1)\, (\lambda_{\max}(\bm{A}) - \lambda_{\min}(\bm{A})) \leq \operatorname{tr}(\bm{A})$
or $(|\textsf{S}| - 1)\,(\sigma_{\max}(\bm{A})^2 - \sigma_{\min}(\bm{A})^2) \leq \lVert \bm{A} \rVert_{\rm F}^2$:
\begin{align*}
    &\mathbb{E} \lVert \bm{A}^{1/2} (\bm{x}_s - \bm{x}_\star) \rVert^2 
    \leq \biggl[1 - \frac{|\textsf{S}|\, \lambda_{\min}(\bm{A})}{\operatorname{tr}(\bm{A}) + (|\textsf{S}| - 1)\, \lambda_{\min}(\bm{A})}
    \biggr] \mathbb{E} \lVert \bm{A}^{1/2} (\bm{x}_{s-1} - \bm{x}_\star) \rVert^2, \\
    \text{or} \quad
    &\mathbb{E} \lVert \bm{x}_s - \bm{x}_\star \rVert^2 
    \leq \biggl[1 - \frac{|\textsf{S}|\,\sigma_{\min}(\bm{A})^2}{\lVert \bm{A} \rVert_{\rm F}^2 + (|\textsf{S}| - 1)\,\sigma_{\min}(\bm{A})^2} \biggr] \mathbb{E} \lVert \bm{x}_{s-1} - \bm{x}_\star \rVert^2.
\end{align*}
See \cref{sec:discussion} for a derivation.
In both randomized coordinate descent and randomized Kaczmarz, the convergence rate is no faster than $(1 - |\textsf{S}| / n)^t$ and it is even slower when $\bm{A}$ is ill-conditioned with singular values of varying sizes.
Hence, SGD methods need to make many passes over the entries of $\bm{A}$ to obtain a high-accuracy solution.

From an information theoretic perspective, the different convergence rates in SGD and RSRI are due to different access models to the matrix.
SGD methods for linear systems are \emph{row access} methods, in which each update depends on a few selected rows of the equation $\bm{A} \bm{x} = \bm{b}$.
RSRI is a \emph{column access} method, in which each update depends on the output vector $\bm{b}$ and a few selected columns of $\bm{A}$.
Row access methods can outperform dense factorization strategies for solving linear systems \cite{derezinski2025randomizedkaczmarzmethodsbeyondkrylov,rathore2025askotchneatsolutionlargescale}, and they are especially useful for solving overdetermined least-squares problems with a small number of unknowns \cite{epperly2025randomizedkaczmarztailaveraging,gilyen2022improved,shao2022faster}.
However, in general, a row access method must access a number of rows proportional to the number of unknowns in the solution vector to produce a high-accuracy solution \cite{epperly2025randomizedkaczmarztailaveraging,chen2019active}.


\subsection{Deterministic sparse approximation} \label{sec:coordinate_pagerank}

The third approach to linear systems, which we call deterministic sparse iteration, takes advantage of information in the residual to determine which entries of the approximation to update.
For example, Cohen, Dahmen, and Devore \cite{cohen2001adaptive} project the linear system $\bm{A} \bm{x} = \bm{b}$ onto an \emph{active set} of indices and solve the reduced system exactly.
To determine the active set, their algorithm trades off between adding new indices corresponding to large-magnitude elements of the residual $\bm{r}(\hat{\bm{x}}) = \bm{b} - \bm{A} \hat{\bm{x}}$ and removing indices corresponding to low-magnitude entries of the approximate solution $\hat{\bm{x}}$.

As another example, the papers \cite{berkhin2006bookmark,andersen2007local} develop a deterministic sparse approximation for the PageRank problem $\bm{x} = \alpha \bm{P} \bm{x} + (1 - \alpha) \bm{s}$.
At each step, their method identifies a single index $i$ for which the residual vector 
\begin{equation*}
    \bm{r}(\hat{\bm{x}}) = 
    (1 - \alpha)\bm{s}
    - (\mathbf{I} - \alpha \bm{P}) \hat{\bm{x}}
\end{equation*}
is largest, and the algorithm updates $\hat{\bm{x}}(i) = \hat{\bm{x}}(i) + \bm{r}(\hat{\bm{x}})(i)$.
With this choice of coordinate-wise update, the residual vector remains nonnegative and the estimates $\hat{\bm{x}}$ are entry-wise increasing.
The algorithm ensures
\begin{equation}
\label{eq:bad_bound}
    \lVert \bm{r}(\hat{\bm{x}}) \rVert_{\infty} < \varepsilon, \quad 
    \text{after a maximum of $t = 1 / \varepsilon$ steps},
\end{equation}
and each step requires accessing just one column of $\bm{P}$ to update the residual.

Compared to the Monte Carlo strategies, the coordinate descent algorithm appears to converge at an accelerated $\mathcal{O}(t^{-1})$ rate, where $t$ denotes the number of update steps.
However, the convergence is measured in a less meaningful norm.
The error bound \cref{eq:bad_bound} only controls the quality of the residual; it does not control the quality of the solution directly.
Additionally, the $\ell_{\infty}$ error norm is smaller than the Euclidean norm.
The conversion factor can be as high as $\sqrt{n}$ given entries of the same magnitude, although it is smaller when the entries in the residual vector are rapidly decaying.

\begin{figure}[t]
    \centering
    \includegraphics[scale=.4]{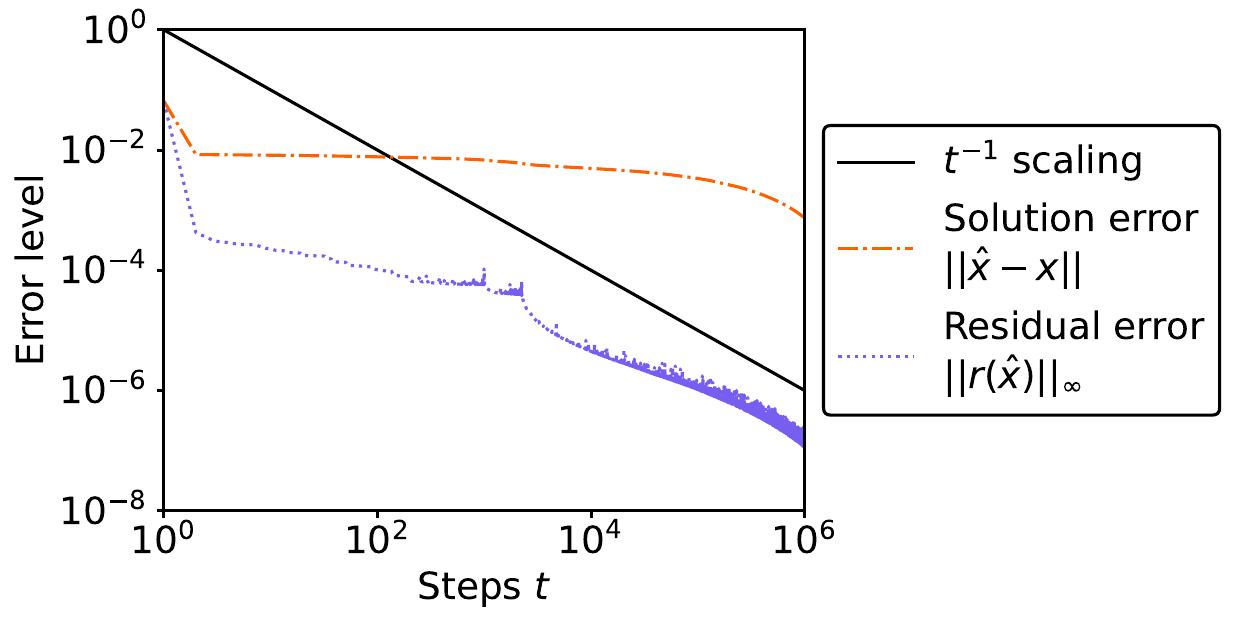}
    \caption{\textbf{(Slow convergence of coordinate descent).} 
    Error $\lVert \hat{\bm{x}} - \bm{x} \rVert$ and residual error $\lVert \bm{r}(\hat{\bm{x}}) \rVert_{\infty}$ for coordinate descent with $t$ update steps, compared to theoretical $t^{-1}$ scaling.
    In the figure, we apply coordinate descent to the Amazon electronics PageRank problem, as documented in \cref{sec:empirical}.}
    \label{fig:slow_deterministic}
\end{figure}
Empirically, coordinate descent leads to slow convergence for large-scale problems, as shown in \cref{fig:slow_deterministic}.
After running the algorithm for $t$ update steps on the Amazon electronics PageRank problem described in \cref{sec:empirical}, we confirm that the residual satisfies
$\lVert \bm{r}(\hat{\bm{x}}) \rVert_{\infty} = \mathcal{O}(t^{-1})$ (purple bottom line).
Nonetheless, the error stagnates when measured in the Euclidean norm $\lVert \hat{\bm{x}} - \bm{x} \rVert$ (orange top line).
The underlying problem is that the residual error is spread out over many entries, so each coordinate update produces just a tiny change in $\hat{\bm{x}}$.


More generally, it remains unclear whether any deterministic sparse iteration method exhibits dimension-independent convergence for a wide class of linear systems.
The paper \cite{cohen2001adaptive} proves specific error bounds for matrices with exponentially decaying off-diagonal entries.
Yet RSRI satisfies more general and powerful error bounds that have not been shown to hold for any deterministic sparse algorithm.

\subsection{Fast randomized iteration} \label{sec:fri}

The work most related to RSRI is ``fast randomized iteration'' (FRI), which was proposed by Lim \& Weare \cite{lim2017fast} and later developed in the papers \cite{greene2019beyond,greene2020improved,greene2022approximating,greene2022full}.
FRI is an approach for speeding up linear or nonlinear fixed-point iterations
\begin{equation*}
    \bm{x}_s = \bm{f}(\bm{x}_{s-1})
\end{equation*}
by repeatedly applying random sparsification
\begin{equation*}
    \bm{x}_s = \bm{f}(\bm{\phi}_s (\bm{x}_{s-1})).
\end{equation*}
Here, $\bm{\phi}_s$ is a random sparsification operator, such as the pivotal sparsification operator (\cref{alg:optimal}).

FRI is an improvement and generalization of the ``full configuration interaction quantum Monte Carlo'' (FCIQMC) approach \cite{booth2009fermion,cleland2010survival} for solving eigenvalue problems in quantum chemistry.
In FCIQMC, random walkers interact in ``a game of life, death, and annihilation'' \cite{booth2009fermion} so that a weighted combination of basis vectors on the walkers' locations
approximates the leading eigenvector of a matrix.
FCIQMC has been applied to matrices as large as $10^{108} \times 10^{108}$ \cite{shepherd2012full}.
Such large matrices are possible because of the combinatorial explosion of basis elements, where each basis element represents an arrangement of electrons into spatial orbitals.
Remarkably, FCIQMC can produce sparse eigenvector approximations for these large systems with eigenvalue errors of just $0.02\%$ using $\mathcal{O}(10^8)$ random walkers \cite{shepherd2012full}.

In contrast to FCIQMC, FRI replaces most of the random operations with deterministic operations in order to improve the efficiency.
As a consequence, FRI produces a solution as accurate as FCIQMC but with a number of time steps that is reduced by a factor of $10^1$--$10^4$ \cite{greene2019beyond,greene2020improved}.
However, FRI has not yet been extended to linear systems.

Here, we extend FRI for the first time to solve linear systems by combining the approach with the deterministic Richardson iteration.
Additionally, we establish the first error bound that explains FRI's faster-than-$1/\sqrt{m}$ convergence rate (\cref{thm:main}), which has been observed in past studies \cite[Fig.~2]{greene2019beyond} but lacked a theoretical explanation.
Since the accelerated convergence rate is the main reason to use FRI in applications,
our work considerably extends and improves the past FRI analyses \cite{lim2017fast,lu2020full,greene2022approximating}.

\section{Random sparsification: design and analysis}
\label{sec:design}

In this section, our goal is to design a sparsification operator
$\bm{\phi}$ that yields small errors
as measured in an appropriate norm.  The results in this section are of interest beyond the specific setting of RSRI.  We address the general question of how to accurately and efficiently approximate a dense vector by a random sparse vector.

Lim \& Weare \cite{lim2017fast} argue that the most meaningful norm for analyzing sparsification error
is the ``triple norm'',
defined for any random vector $\bm{z} \in \mathbb{C}^n$ by 
\begin{equation*}
    \smalliii{\bm{z}} = \biggl(\max_{\lVert \bm{u} \rVert_{\infty} \leq 1} \mathbb{E} \bigl|\bm{u}^{\ast} \bm{z} \bigr|^2\biggr)^{1 / 2}.
\end{equation*}
The triple norm is defined by looking at the worst-case square inner product $\mathbb{E}\bigl|\bm{u}^{\ast} \bm{z}\bigr|^2$ with a vector $\bm{u} \in \mathbb{C}^n$ satisfying $\lVert \bm{u} \rVert_{\infty} \leq 1$.

When we make a random approximation $\hat{\bm{z}} \approx \bm{z}$, we can use the triple norm error $\vertiii{\hat{\bm{z}} - \bm{z}}$ to bound the error of any inner product with any vector $\bm{u} \in \mathbb{C}^n$.
When we approximate the dot product $\bm{u}^* \bm{z}$ with 
$\bm{u}^* \hat{\bm{z}}$, the error is bounded by
\begin{equation*}
    \mathbb{E} \bigl| \bm{u}^* ( \hat{\bm{z}} - \bm{z} ) \bigr|^2
    \leq \lVert \bm{u} \rVert_{\infty}^2 \cdot
    \smalliii{\hat{\bm{z}} - \bm{z}}^2.
\end{equation*}
The triple norm appears in the right-hand side of the error bound.
Additionally, we observe that the triple norm is always larger than the $L^2$ norm
\begin{equation*}
    \mathbb{E} \lVert \hat{\bm{z}} - \bm{z} \rVert^2 
    \leq \smalliii{\hat{\bm{z}} - \bm{z}}^2
\end{equation*}
(see \cref{lem:averaged}).
In this sense, obtaining bounds in the triple norm is more powerful than obtaining bounds in the conventional $L^2$ norm.

As we construct the sparsification operator $\bm{\phi}$, we ideally want the operator to satisfy the following design criteria when applied to any vector $\bm{v} \in \mathbb{C}^n$.
\begin{enumerate}
\item The sparsification should yield at most $m$ nonzero entries, i.e.,
$\lVert \bm{\phi}(\bm{v}) \rVert_0 \leq m$.
\item The sparsification should be unbiased, i.e.,
$\mathbb{E}\bigl[\bm{\phi} (\bm{v})\bigr] = \bm{v}$.
\item The error $\smalliii{\bm{\phi}(\bm{v}) - \bm{v}}$ should be as small as possible, subject to the sparsity and unbiasedness constraints.
\end{enumerate}
In this section, we do not quite succeed in identifying an optimal sparsification operator that satisfies properties 1-3.
However, we identify a ``pivotal'' sparsification operator
that satisfies properties 1-2 and nearly optimizes the triple norm error, up to a factor of $\sqrt{2}$.
See the following new result, which we will prove in \cref{sec:near_optimality}:
\begin{theorem}[Near-optimal error] \label{thm:advantages}
For any fixed vector $\bm{v} \in \mathbb{C}^n$ and any random vector $\bm{z} \in \mathbb{C}^n$ satisfying $\lVert \bm{z} \rVert_0 \leq m$ and $\mathbb{E}\bigl[\bm{z}\bigr] = \bm{v}$,
the pivotal sparsification operator satisfies
\begin{equation*}
    \vertiii{\bm{\phi}_{\rm piv}(\bm{v}) - \bm{v}} \leq \sqrt{2} \vertiii{\bm{Z} - \bm{v}}.
\end{equation*}
Additionally, the pivotal sparsification error is bounded by
\begin{equation}
\label{eq:amazing_bound}
    \vertiii{\bm{\phi}_{\rm piv}(\bm{v}) - \bm{v}}^2 \leq \min_{0 \leq i \leq m} \frac{2}{m - i} \biggl(\sum\nolimits_{j = i + 1}^n |\bm{v}^{\downarrow}(j)|\biggr)^2.
\end{equation}
\end{theorem}
As a result of \cref{thm:advantages}, pivotal sparsification hardly incurs any error if we can arrange the entries $\bm{v}(i)$ so they are decreasing rapidly in magnitude.

The rest of the section is organized as follows.
\Cref{sec:pseudocode} presents pseudocode for pivotal sparsification, \cref{sec:optimality} proves the optimality of pivotal sparsification in the $L^2$ norm,
and \cref{sec:near_optimality} proves near-optimality of pivotal sparsification in the triple norm.

\subsection{Pseudocode}
\label{sec:pseudocode}

\Cref{alg:optimal} describes the pivotal sparsification algorithm, which takes as input a vector $\bm{v} \in \mathbb{C}^n$ and outputs a sparse random vector $\bm{\phi}_{\rm piv}(\bm{v}) \in \mathbb{C}^n$.
The algorithm begins by identifying a set of indices $\textsf{D} \subseteq \{1, \ldots, N\}$ that indicate the largest-magnitude elements of $\bm{v}$.
For simplicity, the pseudocode uses a simple recursive approach for constructing $\textsf{D}$ by adding just one index $i$ at a time.
To improve the efficiency, we can modify \cref{alg:optimal} by adding multiple indices at a time or by preprocessing the input vector to identify the $m$ largest-magnitude entries \cite{floyd64algorithm}.

\begin{algorithm}[t]
\caption{Pivotal sparsification \cite{greene2022approximating}} \label{alg:optimal}
\begin{algorithmic}[1]
\Require Vector $\bm{v} \in \mathbb{C}^n$, sparsity parameter $m$
\Ensure Vector $\bm{\phi}(\bm{v}) \in \mathbb{C}^n$ with no more than $m$ nonzero entries that satisfies $\mathbb{E}\bigl[\bm{\phi}(\bm{v})\bigr] = \bm{v}$, $\mathbb{E} |\bm{\phi}(\bm{v})| = |\bm{v}|$, and $\lVert \bm{\phi}(\bm{v}) \rVert_1 = \lVert \bm{v} \rVert_1$.
\State $\mathsf{D} = \emptyset$
\State $q = 0$
\If{there exists $i \notin \textsf{D}$ such that $|\bm{v}(i)| \geq \frac{1}{m - q} \sum\nolimits_{j \notin \textsf{D}} |\bm{v}(j)|$}
\State $\textsf{D} = \textsf{D} \cup \{i\}$
\State $q = q + 1$
\EndIf
\State Define $\bm{p} \in [0, 1]^n$ with
{\setlength{\abovedisplayskip}{3pt}
\setlength{\belowdisplayskip}{3pt}
\begin{equation*}
    \bm{p}(i) = 
    \begin{cases}
        1, & i \in \textsf{D}, \\
        (m - q) \cdot |\bm{v}(i)| \big\slash \sum\nolimits_{j \notin \textsf{D}} |\bm{v}(j)|, & i \notin \textsf{D}.
    \end{cases}
\end{equation*}
}
\State $\textsf{S} = \texttt{sample}(\bm{p})$
\Comment{Sample using \cref{alg:pivotal}}
\State Return $\bm{\phi}(\bm{v}) \in \mathbb{C}^n$ with
{\setlength{\abovedisplayskip}{3pt}
\setlength{\belowdisplayskip}{3pt}
\begin{equation*}
    \bm{\phi}(\bm{v})(i) =
    \begin{cases}
        \bm{v}(i), & i \in \textsf{D}, \\
        \bm{v}(i) / \bm{p}(i), & i \in \textsf{S}, \\
        0, & i \notin \textsf{D} \cup \textsf{S}.
    \end{cases}
\end{equation*}
}
\end{algorithmic}
\end{algorithm}

After identifying the largest-magnitude entries,
we preserve these entries exactly, that is, we set $\bm{\phi}_{\rm piv}(\bm{v})(i) = \bm{v}(i)$ for $i \in \textsf{D}$.
In contrast, we randomly perturb the smallest-magnitude entries by setting them to zero or by increasing their magnitudes randomly.
We choose which nonzero entries to retain using \emph{pivotal sampling} \cite{deville1998unequal},
a stochastic rounding strategy that takes as input a vector of selection probabilities $\bm{p} \in [0, 1]^n$
and rounds each entry $\bm{p}(i)$ to $0$ or $1$ in an unbiased way, as described in \cref{alg:pivotal}.
Pivotal sampling can be implemented using a single pass through the vector of probabilities.
All of the indices $i \notin \textsf{S}$ that are selected by pivotal sampling are raised to an equal magnitude, which is deterministically chosen to preserve the $\ell_1$-norm of the input, $\lVert \bm{\phi}_{\rm piv}(\bm{v}) \rVert_1 = \lVert \bm{v} \rVert_1$.

\begin{algorithm}[t]
\caption{Pivotal sampling \cite{deville1998unequal}}
\label{alg:pivotal}
\begin{algorithmic}[1]
\Require Vector $\bm{p} \in [0, 1]^n$ with entries that sum to an integer $m = \sum\nolimits_{i=1}^n \bm{p}(i)$
\Ensure Random set of indices $\textsf{S} \subseteq \{1, \ldots, n\}$ with $\mathbb{P}\{i \in \textsf{S}\} = \bm{p}(i)$
\State Set $\textsf{S} = \emptyset$, $b = 0$, $\ell = 0$, $f = 1$
\For{$i = 1, \ldots, m$}
\State $u = \max\{k: b + \sum\nolimits_{j=f}^k \bm{p}(j) < 1\}$
\State Sample $h \in \{\ell, f, f + 1, \ldots, u\}$ with probs. prop. to $(b, \bm{p}(f), \bm{p}(f+1), \ldots, \bm{p}(u))$
\State $b = b + \sum\nolimits_{j=f}^{u+1} \bm{p}(j) - 1$
\State Set $\ell = h$ with probability $(\bm{p}(u+1) - b) / (1 - b)$
\If{$\ell = h$}
\State $\textsf{S} = \textsf{S} \cup \{u+1\}$
\Else
\State $\textsf{S} = \textsf{S} \cup \{h\}$
\EndIf
\State $f = u + 2$
\EndFor
\State Return $\textsf{S}$
\end{algorithmic}
\end{algorithm}


\subsection{Optimality in the \texorpdfstring{$L^2$}{L2} norm}
\label{sec:optimality}

Here, we prove that pivotal sparsification is optimal at controlling error in the $L^2$ norm.
We previously reported a weaker result for pivotal sparsification in \cite[Prop.~5.2]{greene2022approximating}.
By reworking the proof, we are able to obtain a more explicit bound \cref{eq:variance_bound}, which will be used for bounding the RSRI variance in \cref{sec:variance_bounds}.

\begin{proposition}[Optimal $L^2$ error] \label{prop:relaxed}
For any vector $\bm{v} \in \mathbb{C}^n$ and any random vector $\bm{z} \in \mathbb{C}^n$ satisfying $\lVert \bm{z} \rVert_0 \leq m$ and $\mathbb{E}\bigl[\bm{z}\bigr] = \bm{v}$,
it holds that
\begin{equation*}
    \mathbb{E}\lVert \bm{\phi}_{\rm piv}(\bm{v}) - \bm{v} \rVert^2
    \leq \mathbb{E} \lVert \bm{z} - \bm{v} \rVert^2.
\end{equation*}
Additionally, pivotal sparsification satisfies the error bound
\begin{equation}
\label{eq:variance_bound}
    \mathbb{E}\lVert \bm{\phi}_{\rm piv}(\bm{v}) - \bm{v} \rVert^2
    \leq \min_{0 \leq i \leq m} \frac{1}{m - i} \biggl(\sum\nolimits_{j = i + 1}^n |\bm{v}^{\downarrow}(j)|\biggr)^2.
\end{equation}
\end{proposition}
\begin{proof}[Proof of \cref{prop:relaxed}]

The proof is based on an explicit minimization of the square $L^2$ error $\mathbb{E}\lVert \bm{z} - \bm{v} \rVert^2$, subject to the constraints.
To that end, we introduce the random vector $\bm{m} \in \{0, 1\}^n$ with entries 
$\bm{m}(i) = \mathds{1}\{\bm{z}(i) \neq 0\}$.
The entries $\bm{m}(i)$ are thus dependent Bernoulli random variables with success probabilities $\mathbb{E}\bigl[\bm{m}(i)\bigr]$.
Moreover, the constraint 
$\lVert \bm{z} \rVert_0 \leq m$
implies that $\sum\nolimits_{i=1}^n \bm{m}(i) \leq m$,
and the constraint $\mathbb{E} \bigl[\bm{Z}\bigr] = \bm{v}$ implies that
\begin{equation*}
    \bm{v}(i) = \mathbb{E} \bigl[ \bm{z}(i) \bigr]
    = \mathbb{E}\bigl[ \bm{z}(i)\,\big|\, \bm{m}(i) = 1\bigr] \cdot \mathbb{E} \bigl[ \bm{m}(i) \bigr],
    \quad \text{for each } 1 \leq i \leq n.
\end{equation*}
Hence, we can decompose the square $L^2$ error as
\begin{align*}
    \mathbb{E}\lVert \bm{z} - \bm{v} \rVert^2
    &=
    \sum\nolimits_{i = 1}^n \mathbb{E}\bigl|\bm{z}(i)
    - \mathbb{E}\bigl[\bm{z}(i)\,\big|\, \bm{m}(i)\bigr]\bigr|^2
    + \sum\nolimits_{i = 1}^n \mathbb{E}\bigl|\mathbb{E}\bigl[\bm{z}(i)\,\big|\, \bm{m}(i)\bigr] - \bm{v}(i)\bigr|^2 \\
    &=    
    \sum\nolimits_{i=1}^n \mathbb{E}\biggl|\bm{z}(i) - \frac{\bm{m}(i)}{\mathbb{E} \bigl[\bm{m}(i)\bigr] } \bm{v}(i) \biggr|^2 
    + \sum\nolimits_{i=1}^n \mathbb{E}\biggl|\frac{\bm{m}(i)}{\mathbb{E} \bigl[\bm{m}(i)\bigr] } \bm{v}(i) - \bm{v}(i) \biggr|^2.
\end{align*}
We minimize the square $L^2$ error by taking by taking $\bm{z}(i) = \bm{v}(i) \cdot \bm{m}(i) / \,\mathbb{E}\bigl[\bm{m}(i)\bigr]$ for each $1 \leq i \leq n$,
so the first term vanishes and the square $L^2$ error becomes
\begin{equation*}
    \mathbb{E}\lVert \bm{z} - \bm{v} \rVert^2
    = \sum\nolimits_{i=1}^n \mathbb{E}\biggl|\frac{\bm{m}(i)}{\mathbb{E} \bigl[\bm{m}(i)\bigr] } \bm{v}(i) - \bm{v}(i) \biggr|^2
    = \sum\nolimits_{i=1}^n |\bm{v}(i)|^2 \biggl(\frac{1}{\mathbb{E} \bigl[\bm{m}(i)\bigr] } - 1 \biggr).
\end{equation*}
Next, determine the optimal vector of success probabilities $\mathbb{E} \bigl[\bm{m}\big] = \bm{p} \in [0, 1]^n$ by minimizing the objective function
\begin{equation}
\label{eq:convex}
    f(\bm{p})
    = \sum\nolimits_{i=1}^n |\bm{v}(i)|^2 \biggl(\frac{1}{\bm{p}(i)} - 1 \biggr),
\end{equation}
subject to the constraints $\sum\nolimits_{i=1}^n \bm{p}(i) \leq m$ and $\bm{p} \in [0, 1]^n$.

We observe that $\bm{p} \mapsto f(\bm{p})$
is a convex mapping on a closed, convex set.
To find the global minimizer, 
we introduce the Lagrangian function
\begin{equation*}
    L\bigl(\bm{p}, \eta, \bm{\lambda}\bigr)
    = \sum\nolimits_{i=1}^n |\bm{v}(i)|^2 \biggl(\frac{1}{\bm{p}(i)} - 1 \biggr)
    + \eta \biggl(\sum\nolimits_{i=1}^n \bm{p}(i) - m\biggr) + \sum\nolimits_{i=1}^n \bm{\lambda}(i) \bigl(\bm{p}(i) - 1\bigr).
\end{equation*}
By the Karush–Kuhn–Tucker conditions \cite[pg.~321]{nocedal2006numerical},
the minimizer of \cref{eq:convex} must satisfy
$\frac{\partial L}{\partial \bm{p}(i)} = 0$, which leads to
\begin{equation*}
    \bm{p}(i) = \frac{|\bm{v}(i)|}{(\eta + \bm{\lambda}(i))^{1 / 2}}, \quad 1 \leq i \leq n,
\end{equation*}
If $\eta = 0$, the complementarity condition $\bm{\lambda}(i) (\bm{p}(i) - 1) = 0$ implies
that 
\begin{equation}
\label{eq:trivial}
    \bm{p}(i) = 1, \quad \text{for each } 1 \leq i \leq n.
\end{equation}
If $\eta > 0$, the complementarity condition implies that
\begin{equation}
\label{eq:combine_with}
    \bm{p}(i) = \min\biggl\{\frac{|\bm{v}(i)|}{\eta^{1/2}}, 1 \biggr\}, \quad \text{for each } 1 \leq i \leq n.
\end{equation}
In this case, the additional complementarity condition
$\eta\, (\sum\nolimits_{i=1}^n \bm{p}(i) - m) = 0$ implies that
$\sum\nolimits_{i=1}^n \bm{p}(i) = m$.

To better understand the minimizer, introduce a permutation $\sigma_1, \ldots, \sigma_n$ of the indices $1, \ldots, n$ such that
$|\bm{v}(\sigma_1)| \geq \cdots \geq |\bm{v}(\sigma_n)|$,
In light of \cref{eq:trivial,eq:combine_with}, there must be an exact preservation threshold $q \in 0, \ldots, n$ such that
\begin{align*}
    & \bm{p}(\sigma_i) = 1, 
    && i \leq q, && \\
    & \bm{p}(\sigma_i) =
    \frac{(m - q) |\bm{v}(\sigma_i)|}
    {\sum\nolimits_{j = q + 1}^n |\bm{v}(\sigma_j)|} < 1,
    && q + 1 \leq i \leq n. &&
\end{align*}
The objective function \cref{eq:convex} can be rewritten as
\begin{equation}
\label{eq:objective}
    f(q) 
    = \frac{1}{m-q} \Bigl(\sum\nolimits_{i = q + 1}^n |\bm{v}(\sigma_i)|\Bigr)^2
    - \sum\nolimits_{i = q + 1}^n |\bm{v}(\sigma_i)|^2,
\end{equation}
and we can find the optimal $q$ by minimizing \cref{eq:objective} subject to the constraint
\begin{equation*}
    |\bm{v}(\sigma_{q+1})| < \frac{1}{m - q} \sum\nolimits_{i = q + 1}^n |\bm{v}(\sigma_i)|.
\end{equation*}
A direct computation reveals
\begin{equation*}
    f(q + 1) - f(q)
    = \frac{m-q}{m-q-1} \biggl(|\bm{v}(\sigma_{q+1})| - \frac{1}{m-q} \sum\nolimits_{i=q+1}^n |\bm{v}(\sigma_i)|\biggr)^2 \geq 0,
\end{equation*}
wherefore the objective function $f(q)$ is nondecreasing in $q$. The optimal threshold is
\begin{equation}
\label{eq:characterization}
    q_\ast = \min\Bigl\{0 \leq q \leq m:
    |\bm{v}(\sigma_{q+1})| < \frac{1}{m - q} \sum\nolimits_{i = q + 1}^n |\bm{v}(\sigma_i)|
    \Bigr\}.
\end{equation}
This matches the description of pivotal sparsification given in \cref{alg:pivotal}.

Last, to bound the square $L^2$ error of pivotal sparsification, we introduce
\begin{equation*}
    g(q) 
    =
    \frac{1}{m-q} \biggl(\sum\nolimits_{i = q + 1}^n |\bm{v}(\sigma_i)|\biggr)^2
\end{equation*}
which is an upper bound on the objective function $f(q)$ \cref{eq:objective}.
A direct calculation shows that
\begin{equation*}
    g(q) - g(q + 1)
    = \biggl(\frac{1}{m-q-1} \sum\nolimits_{i=q+2}^n |\bm{v}(\sigma_i)|\biggr) \biggl[ |\bm{v}(\sigma_{q+1})| - \frac{1}{m-q} \sum\nolimits_{i=q+1}^n |\bm{v}(\sigma_i)| \biggr].
\end{equation*}
Hence, incrementing $q$ will reduce or leave unchanged the value of $g(q)$ as long as
\begin{equation}
\label{eq:is_satisfied}
    | \bm{v}(\sigma_{q+1}) | \geq \frac{1}{m - q}
    \sum\nolimits_{i=q+1}^n |\bm{v}(\sigma_i)|
\end{equation}
is satisfied.
From the characterization \cref{eq:characterization}, there must be a minimizer $\tilde{q}$ of $g(q)$ satisfying $q_\ast \leq \tilde{q}$.
Since the objective function $f(q)$ is nondecreasing, we conclude
\begin{equation*}
    f(q_\ast) \leq f(\tilde{q}) \leq g(\tilde{q}) = \min_{0 \leq q \leq m} g(q).
\end{equation*}
We rewrite this inequality as
\begin{equation*}
    \mathbb{E}\lVert \bm{\phi}_{\rm piv}(\bm{v}) - \bm{v} \rVert^2
    \leq
    \min_{0 \leq i \leq m} \frac{1}{m - i} \biggl(\sum\nolimits_{j = i + 1}^n |\bm{v}^{\downarrow}(j)|\biggr)^2,
\end{equation*}
to complete the proof.
\end{proof}

\subsection{Near-optimality in the triple norm} \label{sec:near_optimality}

In this section, we prove that pivotal sparsification is nearly optimal at controlling error in the triple norm. This result is entirely new.

To prove the near-optimality result, we need two lemmas.
First, we prove a helpful comparison between the $L^2$ norm and the triple norm.
A similar result was presented in \cite[Eq.~27]{lim2017fast}, but the argument here is simpler.

\begin{lemma}[Difference in norms] \label{lem:averaged}
Fix any random vector $\bm{z} \in \mathbb{C}^n$, and let $\bm{u} \in \mathbb{C}^n$ be an independent random vector with entries that are independent and uniformly distributed on the complex unit circle, so $|\bm{u}(1)| = \cdots = |\bm{u}(n)| = 1$.
It holds that
\begin{equation*}
    \mathbb{E}\lVert \bm{z} \rVert^2 
    = \mathbb{E}\bigl|\bm{u}^\ast \bm{z} \bigr|^2
    \quad \text{and} \quad
    \smalliii{\bm{z}}^2 = \max_{|\bm{v}(1)| = \cdots = |\bm{v}(n)| = 1} \mathbb{E}| \bm{v}^{\ast} \bm{z} |^2.
\end{equation*}
Consequently, $\mathbb{E} \lVert \bm{z} \rVert^2 \leq \smalliii{\bm{z}}^2$.
\end{lemma}
\begin{proof}
By direct calculation
\begin{equation*}
    \mathbb{E} \lVert \bm{z} \rVert^2 
    = \sum\nolimits_{i=1}^n \mathbb{E}\bigl[\overline{\bm{z}(i)} \bm{z}(i)\bigr]
    = \sum\nolimits_{i,j=1}^n \mathbb{E} \bigl[ \overline{\bm{u}(i) \bm{z}(i)} \bm{u}(j) \bm{z}(j)\bigr]
    = \mathbb{E}\bigl|\bm{u}^{\ast} \bm{z}\bigr|^2.
\end{equation*}
This calculation uses the fact that $\bm{u}(i)$ are uncorrelated, mean-zero, variance-one random variables.

Next, since $\bm{v} \mapsto \mathbb{E} |\bm{v}^{\ast} \bm{z} |^2$ is a convex function, the maximum $\max_{\lVert \bm{v} \rVert_{\infty} \leq 1} \mathbb{E} |\bm{v}^\ast \bm{z} |^2$ is achieved on an extreme point of the closed convex set $\{\bm{v} \in \mathbb{C}^n \colon \lVert \bm{v} \rVert_{\infty} \leq 1\}$.
The extreme points are vectors $\bm{v} \in \mathbb{C}^n$ with entries on the complex unit circle, $|\bm{v}(1)| = \cdots = |\bm{v}(n)| = 1$.

The comparison between $\mathbb{E} \lVert \bm{z} \rVert^2$ and $\smalliii{\bm{z}}^2$ then follows, since the inner product of $\bm{z}$ with a random vector $\bm{u}$ has a smaller mean square magnitude than the inner product with a worst-case vector $\bm{v}$.
\end{proof}

Next, we recall a lemma establishing the negative correlations of pivotal sampling, which was proved by Srinivasan \cite{srinivasan2001distributions}.

\begin{lemma}[Negative correlations \cite{srinivasan2001distributions}] \label{lem:negative}
Given a vector of probabilities $\bm{p} \in [0, 1]^n$,
pivotal sampling (\cref{alg:pivotal}) returns a set $\textsf{S} \subseteq \{1, \ldots, n\}$ with negatively correlated selections, i.e.,
\begin{equation}
\label{eq:negative}
    \mathbb{P}\{i, j \in \textsf{S}\} \leq \bm{p}(i) \cdot \bm{p}(j), \quad 1 \leq i,j \leq n.
\end{equation}
\end{lemma}

We are ready to prove \cref{thm:advantages}, which establishes the near-optimality of pivotal sparsification in the triple norm.

\begin{proof}[Proof of \cref{thm:advantages}]
Fix a vector $\bm{v} \in \mathbb{C}^n$ and another vector $\bm{u} \in \mathbb{C}^n$ satisfying
$\lVert \bm{u} \rVert_{\infty} \leq 1$.
Since pivotal sampling leads to negative selection probabilities (\cref{lem:negative}),
the pivotal sparsification operator $\bm{\phi} = \bm{\phi}_{\rm piv}$ satisfies
\begin{equation*}
    \mathbb{E} \biggl[ \frac{\phi(\bm{v})(i)}{\bm{v}(i)}
    \cdot \frac{\phi(\bm{v})(j)}{\bm{v}(j)}\biggr] \leq 1, \quad \text{if } \bm{v}(i) \neq 0 \text{ and } \bm{v}(j) \neq 0.
\end{equation*}
Now define the index sets
\begin{equation*}
    \textsf{P} = \{1 \leq i \leq n \,|\, \operatorname{Re}\{\overline{\bm{u}(i)} \bm{v}(i)\} > 0\} \quad \text{and} \quad
    \textsf{N} = \{1 \leq i \leq n \,|\, \operatorname{Re}\{\overline{\bm{u}(i)} \bm{v}(i)\} < 0\}.
\end{equation*}
For each pair of indices $i,j \in \textsf{P}$ or $i,j \in \textsf{N}$, it follows
\begin{align*}
    & \mathbb{E} \bigl[\operatorname{Re}\{ \overline{\bm{u}(i)} \phi(\bm{v})(i) \} 
    \cdot \operatorname{Re}\{ \overline{\bm{u}(j)} \phi(\bm{v})(j) \}\bigr] \\
    & = \operatorname{Re}\{ \overline{\bm{u}(i)} \bm{v}(i) \} 
    \cdot \operatorname{Re}\{ \overline{\bm{u}(j)} \bm{v}(j) \} 
    \cdot \mathbb{E} \biggl[ \frac{\phi(\bm{v})(i)}{\bm{v}(i)}
    \cdot \frac{\phi(\bm{v})(j)}{\bm{v}(j)}\biggr] \\ & \leq \operatorname{Re}\{ \overline{\bm{u}(i)} \bm{v}(i) \} 
    \cdot \operatorname{Re}\{ \overline{\bm{u}(j)} \bm{v}(j) \}.
\end{align*}
Therefore the following covariance is negative:
\begin{equation}
\label{eq:parity}
    \operatorname{Cov} \bigl[\operatorname{Re}\{ \overline{\bm{u}(i)} \phi(\bm{v})(i) \}, \operatorname{Re}\{ \overline{\bm{u}(j)} \phi(\bm{v})(j) \}\bigr] \leq 0.
\end{equation}
The negative covariance relation \cref{eq:parity} allows us to calculate
\begin{align*}
    & \operatorname{Var}\bigl[\operatorname{Re}\{\bm{u}^{\ast} \bm{\phi}(\bm{v})\}\bigr] 
    = \operatorname{Var}\Bigl[\sum\nolimits_{i=1}^n \operatorname{Re}\{ \overline{\bm{u}(i)} \phi(\bm{v})(i) \} \Bigr] \\
    &\leq 2 \operatorname{Var}\Bigl[\sum\nolimits_{i \in \textsf{P}} \operatorname{Re}\{ \overline{\bm{u}(i)} \bm{\phi}(\bm{v})(i) \} \Bigr]
    + 2 \operatorname{Var}\Bigl[\sum\nolimits_{i \in \textsf{N}} \operatorname{Re}\{ \overline{\bm{u}(i)} \bm{\phi}(\bm{v})(i) \} \Bigr] \\
    &\leq 2 \sum\nolimits_{i \in \textsf{P}} \operatorname{Var}\bigl[ \operatorname{Re}\{ \overline{\bm{u}(i)} \bm{\phi}(\bm{v})(i) \} \bigr]
    + 2\sum\nolimits_{i \in \textsf{N}} \operatorname{Var}\bigl[ \operatorname{Re}\{ \overline{\bm{u}(i)} \bm{\phi}(\bm{v})(i) \} \bigr] \\
    &= 2 \sum\nolimits_{i=1}^n \operatorname{Var}\bigl[ \operatorname{Re}\{ \overline{\bm{u}(i)} \bm{\phi}(\bm{v})(i) \} \bigr].
\end{align*}
Similarly, we calculate
\begin{equation*}
    \operatorname{Var}\bigl[\operatorname{Im}\{\bm{u}^{\ast} \bm{\phi}(\bm{v})\}\bigr]
    \leq 2 \sum\nolimits_{i=1}^n \operatorname{Var}\bigl[ \operatorname{Im}\{ \overline{\bm{u}(i)} \bm{\phi}(\bm{v})(i) \} \bigr].
\end{equation*}
By combining the real and imaginary parts,
\begin{align*}
    \operatorname{Var}\bigl[ \bm{u}^{\ast} \bm{\phi}(\bm{v}) \bigr] 
    & = \operatorname{Var}\bigl[\operatorname{Re}\{\bm{u}^{\ast} \bm{\phi}(\bm{v})\}\bigr]
    + \operatorname{Var}\bigl[\operatorname{Im}\{\bm{u}^{\ast} \bm{\phi}(\bm{v})\}\bigr] \\
    & \leq 2 \sum\nolimits_{i=1}^n \operatorname{Var}\bigl[ \operatorname{Re}\{ \overline{\bm{u}(i)} \bm{\phi}(\bm{v})(i) \} \bigr]
    + 2 \sum\nolimits_{i=1}^n \operatorname{Var}\bigl[ \operatorname{Im}\{ \overline{\bm{u}(i)} \bm{\phi}(\bm{v})(i) \} \bigr] \\
    &= 2 \sum\nolimits_{i=1}^n \operatorname{Var}\bigl[ \overline{\bm{u}(i)} \bm{\phi}(\bm{v})(i) \bigr] \\
    & \leq 2 \sum\nolimits_{i=1}^n \operatorname{Var}\bigl[ \bm{\phi}(\bm{v})(i) \bigr] = 2 \, \mathbb{E} \bigl\lVert \bm{\phi}(\bm{v}) - \bm{v} \bigr\rVert^2.
\end{align*}
The last line bounds the variance using the fact that $\lVert \bm{u} \rVert_{\infty} \leq 1$.

Last, for any random vector $\bm{z} \in \mathbb{C}^n$ satisfying $\lVert \bm{z} \rVert_0 \leq m$ and $\mathbb{E}\bigl[\bm{z}\bigr] = \bm{v}$, \cref{prop:relaxed,lem:averaged} allow us to calculate
\begin{equation*}
    2 \, \mathbb{E} \bigl\lVert \bm{\phi}(\bm{v}) - \bm{v} \bigr\rVert^2
    \leq 2 \, \mathbb{E} \lVert \bm{Z} - \bm{v} \rVert^2 
    \leq 2 \, \smalliii{\bm{Z} - \bm{v}}^2.
\end{equation*}
We combine with the variance bound from \cref{prop:relaxed} to show
\begin{equation*}
    \vertiii{\bm{\phi}(\bm{v}) - \bm{v}}^2 
    \leq \min_{0 \leq i \leq m} \frac{2}{m - i} \biggl(\sum\nolimits_{j = i + 1}^n |\bm{v}^{\downarrow}(i)|\biggr)^2.
\end{equation*}
This verifies the error bound \cref{eq:amazing_bound} and completes the proof.
\end{proof}

Last, by examining the proof of \cref{thm:advantages}, we record a simple corollary which allows us to sometimes remove a factor of two from the error bound.
\begin{corollary}[Nonnegative sparsification bound] \label{cor:sharper}
For any vectors $\bm{u}, \bm{v} \in \mathbb{C}^n$ satisfying $\operatorname{Re}\{\overline{\bm{u}(i)} \bm{v}(i)\} \geq 0$ and $\operatorname{Im}\{\overline{\bm{u}(i)} \bm{v}(i) \} \geq 0$ for all $1 \leq i \leq n$,
the pivotal sparsification error is bounded by
\begin{equation*}
    \mathbb{E} \bigl| \bm{f}^\ast \bigl(\bm{\phi}_{\rm piv}(\bm{v}) - \bm{v}\bigr) \bigr|^2 
    \leq \min_{0 \leq i \leq m} \frac{1}{m - i} \biggl(\sum\nolimits_{j = i + 1}^n |\bm{v}^{\downarrow}(i)|\biggr)^2.
\end{equation*}
\end{corollary}

\section{RSRI error bounds} \label{sec:RSRI_error}

Our main goal in this section is to prove \cref{thm:extended},
which is a stronger version of \cref{thm:main} from the introduction.

\begin{theorem}[Extended error bound] \label{thm:extended}
Suppose RSRI with sparsity level $m$ is applied to an $n \times n$ linear system $\bm{A} \bm{x} = \bm{b}$, and set $\bm{G} = \mathbf{I} - \bm{A}$.
Then, RSRI returns a solution $\overline{\bm{x}}_t$ satisfying the bias-variance formula
\begin{equation}
\label{eq:formula}
    \vertiii{ \bm{A} \overline{\bm{x}}_t - \bm{b} }^2
    \leq \underbrace{\bigl\lVert \bm{A} \,\mathbb{E}\bigl[\overline{\bm{x}}_t\bigr] - \bm{b} \bigr\rVert_1^2}_{\textup{bias}^2}
    + \underbrace{\vertiii{ \bm{A} \overline{\bm{x}}_t - \bm{A} \,\mathbb{E}\bigl[ \overline{\bm{x}}_t \bigr] }^2}_{\textup{variance}}.
\end{equation}
The square bias is bounded by
\begin{equation*}
    \textup{bias}^2 \leq 
    \biggl(\frac{2 \sup_{s \geq 0} \lVert \bm{G}^s \rVert_1 \cdot \bigl\lVert \bm{G}^{t_{\rm b}} \bm{x}_{\star} \bigr\rVert_1}{t - t_{\rm b}}\biggr)^2,
\end{equation*}
where $\bm{x}_\star$ is the exact solution.
If $\bm{G}$ is a strict $1$-norm contraction, $\lVert \bm{G} \rVert_1 < 1$, the variance is bounded by
\begin{equation*}
    \textup{variance}
    \leq \frac{8 t}{(t - t_{\rm b})^2} \cdot 
    \frac{1}{m} \biggl(\frac{\lVert \bm{b} \rVert_1}{1 - \lVert \bm{G} \rVert_1}\biggr)^2.
\end{equation*}
Alternately, if $m \geq m_{\bm{G}} = \sum\nolimits_{s=0}^\infty \lVert | \bm{G} |^s \rVert_1^2$, the variance is bounded by
\begin{equation*}
    \textup{variance}
    \leq 
    \frac{8 t \sup_{s \geq 0} \lVert \bm{G}^s \rVert_1^2}{(t - t_{\rm b})^2}
    \cdot \min_{0 \leq i \leq m - m_{\bm{G}}} \frac{1}{m - m_{\bm{G}} - i} 
    \biggl(\sum\nolimits_{j = i + 1}^n \tilde{\bm{x}}^{\downarrow}(i) \biggr)^2.
\end{equation*}
Here, $\tilde{\bm{x}} \in \mathbb{R}^n$ is the solution to the regularized linear system
$\tilde{\bm{A}} \tilde{\bm{x}} = \tilde{\bm{b}}$, where $\tilde{\bm{A}} = \mathbf{I} - |\bm{G}|$, $\tilde{\bm{b}} = |\bm{b}|$, and $\tilde{\bm{x}}^{\downarrow} \in \mathbb{R}^n$ is the decreasing rearrangement of $\tilde{\bm{x}}$.
\end{theorem}
\begin{proof}
For any $\bm{u} \in \mathbb{C}^n$, we observe
\begin{equation*}
    \mathbb{E} \bigl| \bm{u}^\ast \bigl( \bm{A} \overline{\bm{x}}_t - \bm{b} \bigr)\bigr|^2
    = \bigl| \bm{u}^\ast \bigl( \bm{A}\, \mathbb{E}\bigl[\overline{\bm{x}}_t\bigr] - \bm{b} \bigr)\bigr|^2
    + \mathbb{E} \bigl| \bm{u}^\ast \bigl( \bm{A} \overline{\bm{x}}_t - \bm{A}\, \mathbb{E}\bigl[ \overline{\bm{x}}_t\bigr] \bigr)\bigr|^2.
\end{equation*}
By taking the supremum over all $\bm{u}$ satisfying $\lVert \bm{u} \rVert_{\infty} \leq 1$,
we establish the bias-variance formula \cref{eq:formula}.
We will prove the bias bounds in \cref{sec:bias_bounds} and the variance bounds in \cref{sec:variance_bounds}.
\end{proof}

\Cref{thm:extended} demonstrates that raising the sparsity threshold $m$ has multiple benefits.
First, raising $m$ leads to faster-than-$1/\sqrt{m}$ convergence if the regularized solution $\tilde{\bm{x}}$ has rapidly decaying entries.
Second, raising $m$ extends RSRI to any system in which the largest eigenvalue of $|\bm{G}|$ is strictly less than one.
Indeed, RSRI is guaranteed to converge as $t \rightarrow \infty$, if we set
\begin{equation*}
    m \geq m_{\bm{G}} = \sum\nolimits_{s=0}^\infty \lVert | \bm{G} |^s \rVert_1^2.
\end{equation*}
In comparison, the deterministic Richardson iteration (\cref{alg:ji_classical}) converges if the spectral radius of $\bm{G}$ is strictly less than one.
The spectral radius of $\bm{G}$ is always bounded from above by the largest eigenvalue of $|\bm{G}|$ \cite[Thm.~8.3.2]{horn2012matrix}.

To measure error, \cref{thm:extended} uses the triple norm, which is more useful than the $L^2$ norm.
Thanks to  our use of the triple norm, we are able to transfer our results to the PageRank problem and verify \cref{prop:pagerank}.

\begin{proof}[Proof of \cref{prop:pagerank}]
In the PageRank problem $\bm{A} \bm{x} = \bm{b}$ where $\bm{A} = \mathbf{I} - \alpha \bm{P}$ and $\bm{b} = (1 - \alpha) \bm{s}$, we observe that $\bm{A}^{-1} = \sum\nolimits_{s=0}^{\infty} \alpha^s \bm{P}^s$ has nonnegative columns that sum to $\sum_{s=0}^{\infty} \alpha^s = 1/(1 - \alpha)$.
Therefore, $\lVert \bm{A}^{-1} \rVert_1 = 1/(1 - \alpha)$, and we calculate
\begin{align*}
    \vertiii{ \overline{\bm{x}}_t - \bm{x}_\star }^2
    &= \max_{\lVert \bm{u} \rVert_{\infty} \leq 1} \mathbb{E} \bigl \lVert \bm{u}^\ast \bigl(\overline{\bm{x}}_t - \bm{x}_\star \bigr) \bigr \rVert^2 \\
    &\leq \lVert \bm{A}^{-1} \rVert_1^2 \max_{\lVert \bm{u} \rVert_{\infty} \leq 1} \mathbb{E} \bigl \lVert \bm{u}^\ast \bigl( \bm{A} \overline{\bm{x}}_t - \bm{b} \bigr) \bigr \rVert^2
    = \lVert \bm{A}^{-1} \rVert_1^2 
    \vertiii{ \bm{A} \overline{\bm{x}}_t - \bm{b} }^2 \\
    &= \frac{1}{(1 - \alpha)^2} 
    \vertiii{ \bm{A} \overline{\bm{x}}_t - \bm{b} }^2.
\end{align*}
To verify \cref{prop:pagerank}, we apply the bias and variance bounds from \cref{thm:extended} and use the fact that the regularized system $\tilde{\bm{A}} \tilde{\bm{x}} = \tilde{\bm{b}}$ is the same as the original system $\bm{A} \bm{x} = \bm{b}$.
\end{proof}

\subsection{Bias bound}
\label{sec:bias_bounds}

In this section, we bound RSRI's bias.
RSRI has the same bias that would result from deterministic Richardson iteration (\cref{alg:ji_classical}) after averaging the iterates $\bm{x}_{t_{\rm b}}, \ldots, \bm{x}_{t - 1}$, and we prove the following bound.

\begin{proposition}[Bias bound]
\label{prop:bias_bound}
Suppose RSRI is applied to a linear system $\bm{A} \bm{x} = \bm{b}$, and set $\bm{G} = \mathbf{I} - \bm{A}$.
RSRI returns an approximation $\overline{\bm{x}}_t$ of the exact solution $\bm{x}_{\star}$ with bias bounded by
\begin{equation*}
    \bigl\lVert \bm{A}\, \mathbb{E}\bigl[\overline{\bm{x}}_t\bigr] - \bm{b} \bigr\rVert_1
    \leq \frac{2 \sup_{s \geq 0} \lVert \bm{G}^s \rVert_1 \cdot \bigl\lVert \bm{G}^{t_{\rm b}} \bm{x}_{\star} \bigr\rVert_1}{t - t_{\rm b}}.
\end{equation*}
\end{proposition}
\begin{proof}
We can determine the bias of RSRI from the recursion
\begin{equation*}
    \mathbb{E} \bigl[ \bm{x}_s \bigr] = \bm{G}\, \mathbb{E} \bigl[ \bm{x}_{s-1} \bigr] + \bm{b}.
\end{equation*}
The recursion leads to
\begin{equation*}
    \mathbb{E} \bigl[\bm{x}_s\bigr] = \sum\nolimits_{r=0}^{s-1} \bm{G}^r \bm{b} = \bm{x}_{\star} - \bm{G}^s \bm{x}_{\star},
\end{equation*}
where we have substituted $\bm{b} = \bm{A} \bm{x}_\star = (\mathbf{I} - \bm{G}) \bm{x}_{\star}$.
Using
$\overline{\bm{x}}_t = \frac{1}{t - t_{\rm b}} \sum\nolimits_{s=t_{\rm b}}^{t - 1} \bm{x}_s$ and $\bm{A} = \mathbf{I} - \bm{G}$, we calculate
\begin{align*}
    \bm{A}\, \mathbb{E} \bigl[\overline{\bm{x}}_t \bigr] - \bm{b}
    &= (\mathbf{I} - \bm{G}) \Bigl(\frac{1}{t - t_{\rm b}} \sum\nolimits_{s=t_{\rm b}}^{t - 1} \mathbb{E} \bigl[ \bm{x}_s \bigr] \Bigr) - (\mathbf{I} - \bm{G}) \bm{x}_{\star} \\
    &= (\mathbf{I} - \bm{G}) \Bigl(-\frac{1}{t - t_{\rm b}} \sum\nolimits_{s=t_{\rm b}}^{t - 1} \bm{G}^s \bm{x}_{\star} \Bigr) \\
    &= \frac{1}{t - t_{\rm b}} (\bm{G}^t - \bm{G}^{t_{\rm b}}) \bm{x}_{\star}.
\end{align*}
We apply the series of upper bounds
\begin{align}
    \bigl\lVert \bm{A}\, \mathbb{E}\bigl[ \overline{\bm{x}}_t\bigr] - \bm{b} \bigr\rVert_1
    &= \frac{1}{t - t_{\rm b}}  \bigl\lVert \bigl(\bm{G}^{t} - \bm{G}^{t_{\rm b}}\bigr) \bm{x}_{\star} \bigr\rVert_1 \\
\label{eq:submultiplicativity}
    &\leq \frac{1}{t - t_{\rm b}} \lVert \bm{G}^{t -t_{\rm b}} - \mathbf{I} \rVert_1 \bigl\lVert \bm{G}^{t_{\rm b}} \bm{x}_{\star} \bigr\rVert_1 \\
\label{eq:subadditivity}
    &\leq \frac{2}{t - t_{\rm b}} \sup_{s \geq 0} \lVert \bm{G}^s \rVert_1 \bigl\lVert \bm{G}^{t_{\rm b}} \bm{x}_{\star} \bigr\rVert_1,
\end{align}
where
\cref{eq:submultiplicativity} is due to the submultiplicativity of the matrix 1-norm
and \cref{eq:subadditivity} is due to the subadditivity of the matrix 1-norm.
\end{proof}

\subsection{Variance bounds} \label{sec:variance_bounds}

In this section, we bound the variance of RSRI.
To that end, we will need to bound the maximum square sparsification error
\begin{equation}
\label{eq:to_bound}
    \sup_{t \geq 0} \vertiii{\bm{\phi}_{t+1}(\bm{x}_t) - \bm{x}_t}^2.
\end{equation}
We first analyze \cref{eq:to_bound} in the simple case that $\bm{G}$ is a strict one-norm contraction.

\begin{proposition}[Strict one-norm contraction bound] \label{prop:strict}
Suppose RSRI with sparsity level $m$ is applied to a linear system $\bm{A} \bm{x} = \bm{b}$, 
and set $\bm{G} = \mathbf{I} - \bm{A}$.
If $\lVert \bm{G} \rVert_1 < 1$, the maximum square sparsification error is bounded by
\begin{equation}
\label{eq:bound_1}
    \sup_{t \geq 0} \vertiii{\bm{\phi}_{t+1}(\bm{x}_t) - \bm{x}_t}^2
    \leq \frac{2}{m} \biggl(\frac{\lVert \bm{b} \rVert_1}{1 - \lVert \bm{G} \rVert_1}\biggr)^2.
\end{equation}
\end{proposition}
\begin{proof}
The RSRI iterates are generated according to the recursion
\begin{equation*}
    \bm{x}_s = \bm{G} \bm{\phi}_s (\bm{x}_{s-1}) + \bm{b},
\end{equation*}
which leads to the recursive inequality
\begin{equation*}
    \lVert \bm{x}_s \rVert_1 
    \leq \lVert \bm{G} \rVert_1 \lVert \bm{\phi}_s(\bm{x}_{s-1}) \rVert_1 + \lVert \bm{b} \rVert_1
    = \lVert \bm{G} \rVert_1 \lVert \bm{x}_{s-1} \rVert_1 + \lVert \bm{b} \rVert_1,
\end{equation*}
where we use the fact that $\lVert \bm{\phi}_s(\bm{x}_{s-1}) \rVert_1 = \lVert \bm{x}_{s-1} \rVert_1$ with probability one.
The recursive inequality leads to the upper bound
\begin{equation*}
    \lVert \bm{x}_t \rVert_1 \leq \sum\nolimits_{s=0}^t \lVert \bm{G} \rVert_1^s \lVert \bm{b} \rVert_1 \leq \frac{\lVert \bm{b} \rVert_1}{1 - \lVert \bm{G} \rVert_1}.
\end{equation*}
Last, we apply \cref{thm:advantages} to calculate
\begin{equation*}
    \vertiii{\bm{\phi}_{t+1}(\bm{x}_t) - \bm{x}_t}^2 
    \leq \frac{2}{m} \mathbb{E} \lVert \bm{x}_t \rVert_1^2
    \leq \frac{2}{m} \biggl(\frac{\lVert \bm{b} \rVert_1}{1 - \lVert \bm{G} \rVert_1}\biggr)^2,
\end{equation*}
which confirms \cref{eq:bound_1} and completes the proof.
\end{proof}

To obtain a more powerful bound on \cref{eq:to_bound}, we need a lemma that controls the absolute values of the entries of the RSRI iterates $\bm{x}_0, \bm{x}_1, \ldots$.

\begin{lemma}[Stability lemma]
\label{lem:stability}
Suppose RSRI is applied to a linear system $\bm{A} \bm{x} = \bm{b}$, and set $\bm{G} = \mathbf{I} - \bm{A}$.
Then, the RSRI iterates $\bm{x}_0, \bm{x}_1, \ldots$ satisfy
\begin{equation}
\label{eq:entry_wise}
    \mathbb{E} \bigl|\bm{x}_t \bigr|
    \leq \sum\nolimits_{s=0}^{t-1} |\bm{G}|^s |\bm{b}|, \quad \text{for each } t \geq 0.
\end{equation}
\end{lemma}
\begin{proof}
At any iteration $s \geq 0$, the RSRI iterates satisfy
\begin{equation*}
    |\bm{x}_s| \leq |\bm{G}| |\bm{\phi}_{s}(\bm{x}_{s-1})| + |\bm{b}|.
\end{equation*}
Taking expectations and using the fact that $\mathbb{E} \bigl|\bm{\phi}_s(\bm{x}_{s-1})\bigr| = \mathbb{E} \bigl|\bm{x}_{s-1}\bigr|$, we obtain the recursion
\begin{equation*}
    \mathbb{E} |\bm{x}_s | 
    \leq |\bm{G}|\, \mathbb{E} \bigl|\bm{\phi}_s (\bm{x}_{s-1}) \bigr| + |\bm{b}|
    = |\bm{G}| \, \mathbb{E} \bigl|\bm{x}_{s-1} \bigr| + |\bm{b}|
\end{equation*}
Since $\bm{x}_0 = \bm{0}$, this recursion validates the error bound \cref{eq:entry_wise}.
\end{proof}

Last, we establish a more powerful error bound for the maximum square sparsification error \cref{eq:to_bound}.

\begin{proposition}[More powerful bound] \label{prop:large}
Suppose RSRI is applied to an $n \times n$ linear system $\bm{A} \bm{x} = \bm{b}$, and set $\bm{G} = \mathbf{I} - \bm{A}$.
If the sparsity level satisfies $m \geq m_{\bm{G}} = \sum\nolimits_{s=0}^\infty \lVert |\bm{G}|^s \rVert_1^2$,
the maximum square sparsification error is bounded by
\begin{equation*}
    \sup_{t \geq 0} \vertiii{\bm{\phi}_{t+1}(\bm{x}_{t}) - \bm{x}_{t}}^2 
    \leq \min_{0 \leq i \leq m - m_{\bm{G}}} \frac{2}{m - m_{\bm{G}} - i} 
    \biggl(\sum\nolimits_{j = i + 1}^n \tilde{\bm{x}}^{\downarrow}(i)\biggr)^2.
\end{equation*}
Here, $\tilde{\bm{x}} \in \mathbb{R}^n$ is the solution to the regularized linear system
$\tilde{\bm{A}} \tilde{\bm{x}} = \tilde{\bm{b}}$, where $\tilde{\bm{A}} = \mathbf{I} - |\bm{G}|$, $\tilde{\bm{b}} = |\bm{b}|$, and $\tilde{\bm{x}}^{\downarrow} \in \mathbb{R}^n$ is the decreasing rearrangement of $\tilde{\bm{x}}$.
\end{proposition}
\begin{proof}
Fix a set $\textsf{E} \subseteq \{1, \ldots, N\}$ with $|\textsf{E}| \leq m - m_{\bm{G}}$, and let $\bm{u} \in \{0, 1\}^n$ be defined by $\bm{u}(i) = 0$ for $i \in \textsf{E}$ and $\bm{u}(i) = 1$ for $i \notin \textsf{E}$.
Then, \cref{thm:advantages} guarantees
\begin{equation}
\label{eq:combine_me}
    \vertiii{\bm{\phi}_{t+1}(\bm{x}_{t}) - \bm{x}_{t}}^2
    \leq \frac{2}{m - |\textsf{E}|} \cdot \mathbb{E} \bigl| \bm{u}^\ast |\bm{x}_t|\bigr|^2
\end{equation}
We will proceed to derive an error bound on $\mathbb{E} \bigl| \bm{u}^\ast |\bm{x}_t|\bigr|^2$.

For any $0 \leq s \leq t - 1$, we make the following calculation, which is based on expanding the square and using \cref{lem:stability} to calculate the expectations:
\begin{equation}
\label{eq:squares}
\begin{aligned}
    & \mathbb{E} \bigl| \bm{u}^\ast |\bm{G}|^s |\bm{x}_{t-s}|\bigr|^2 - \mathbb{E} \bigl| \bm{u}^\ast |\bm{G}|^{s+1} |\bm{\phi}_{t-s}(\bm{x}_{t-s-1}) | \bigr|^2 \\
    &\leq \mathbb{E} \bigl| \bm{u}^\ast |\bm{G}|^{s+1} |\phi_{t-s}(\bm{x}_{t-s-1})| + \bm{u}^\ast |\bm{G}|^s | |\bm{b} | \bigr|^2 - \mathbb{E} \bigl| \bm{u}^\ast |\bm{G}|^{s+1} |\bm{\phi}_{t-s}(\bm{x}_{t-s-1}) | \bigr|^2 \\
    &\leq 2
    \Bigl(\sum\nolimits_{r = s+1}^{t-1} \bm{u}^\ast |\bm{G}|^r |\bm{b}|\Bigr) \bigl(\bm{u}^\ast |\bm{G}|^s |\bm{b}|\bigr)
    + \bigl(\bm{u}^\ast |\bm{G}|^s |\bm{b}|\bigr)^2 \\
    &= \Bigl(\sum\nolimits_{r = s}^{t-1} \bm{u}^\ast |\bm{G}|^r |\bm{b}|\Bigr)^2
    - \Bigl(\sum\nolimits_{r = s+1}^{t-1} \bm{u}^\ast |\bm{G}|^r |\bm{b}|\Bigr)^2.
\end{aligned}
\end{equation}
Next, we introduce a vector $\bm{w} \in \mathbb{C}^n$ with entries
\begin{equation*}
    \bm{w}(i) = 
    \begin{cases} \bm{u}^\ast |\bm{G}|^{s+1} \bm{e}_i \cdot \frac{\bm{x}_{t-s}(i)}{|\bm{x}_{t-s}(i)|}, & \bm{x}_{t-s}(i) \neq 0, \\
    0, & \bm{x}_{t-s}(i) = 0.
    \end{cases}
\end{equation*}
and observe that
\begin{equation*}
    \lVert \bm{w} \rVert_{\infty}
    =
    \max_{1 \leq i \leq n} 
    \bm{u}^* |\bm{G}|^{s+1} \bm{e}_i
    \leq
    \lVert |\bm{G}|^{s+1} \rVert_1.
\end{equation*}
We make the following calculation, which is based on the conditional expectation $\mathbb{E}\bigl[|\bm{\phi}_{t-s}(\bm{x}_{t-s-1})|\,|\,\bm{x}_{t-s-1}\bigr] = |\bm{x}_{t-s-1} |$ and an application of \cref{cor:sharper} to bound the variance:
\begin{equation}
\label{eq:truncation}
\begin{aligned}
    & \mathbb{E} \bigl| \bm{u}^\ast |\bm{G}|^{s+1} |\bm{\phi}_{t-s}(\bm{x}_{t-s-1}) | \bigr|^2
    - \mathbb{E} \bigl| \bm{u}^\ast |\bm{G}|^{s+1} |\bm{x}_{t-s-1} | \bigr|^2 \\
    &= \mathbb{E} \bigl| \bm{u}^\ast |\bm{G}|^{s+1} \bigl[|\bm{\phi}_{t-s}(\bm{x}_{t-s-1} )| - |\bm{x}_{t-s-1}| \bigr] \bigr|^2 \\
    &= \mathbb{E} \bigl| \bm{w}^\ast \bigl[\bm{\phi}_{t-s} (\bm{x}_{t-s-1} ) - \bm{x}_{t-s-1}\bigr] \bigr|^2 \\
    &\leq \frac{\lVert |\bm{G}|^{s+1} \rVert_1^2}{m - |\textsf{E}|} \cdot \mathbb{E} \bigl| \bm{u}^\ast \bigl|\bm{x}_{t-s-1}\bigr| \bigr|^2
\end{aligned}
\end{equation}
Adding \cref{eq:squares} to \cref{eq:truncation}, we find that
\begin{multline}
\label{eq:sum_me}
    \mathbb{E} \bigl| \bm{u}^\ast |\bm{G}|^s |\bm{x}_{t-s}|\bigr|^2 - \mathbb{E} \bigl| \bm{u}^\ast |\bm{G}|^{s+1} |\bm{x}_{t-s-1}|\bigr|^2 \\
    \leq \Bigl(\sum\nolimits_{r = s}^{t-1} \bm{u}^\ast |\bm{G}|^r |\bm{b}|\Bigr)^2
    - \Bigl(\sum\nolimits_{r = s+1}^{t-1} \bm{u}^\ast |\bm{G}|^r |\bm{b}|\Bigr)^2 + \frac{\lVert |\bm{G}|^{s+1} \rVert_1^2}{m - |\textsf{E}|} 
    \cdot \mathbb{E} \bigl| \bm{u}^\ast \bigl|\bm{x}_{t-s-1}\bigr| \bigr|^2.
\end{multline}
We sum over \cref{eq:sum_me} for $s = 0, 1, \ldots, t-1$ to obtain the recursion
\begin{equation*}
    \mathbb{E} \bigl| \bm{u}^\ast |\bm{x}_t|\bigr|^2 
    \leq \Bigl(\sum\nolimits_{r = 0}^{t-1} \bm{u}^\ast |\bm{G}|^r |\bm{b}|\Bigr)^2
    + \frac{1}{m - |\textsf{E}|}
    \sum\nolimits_{s=1}^t \lVert |\bm{G}|^s \rVert_1^2 \, \mathbb{E} \bigl| \bm{u}^\ast \bigl|\bm{x}_{t-s}\bigr| \bigr|^2.
\end{equation*}
The recursion implies that
\begin{equation*}
    \sup_{t \geq 0} \mathbb{E} \bigl| \bm{u}^\ast |\bm{x}_t|\bigr|^2
    \leq \frac{\bigl(\sum\nolimits_{s = 0}^\infty \bm{u}^\ast |\bm{G}|^s |\bm{b}|\bigr)^2}{1 - \dfrac{1}{m - |\textsf{E}|}
    \sum\nolimits_{s=0}^\infty \lVert | \bm{G} |^s \rVert_1^2}
    = \frac{\bigl(\sum\nolimits_{i \notin \textsf{E}} \tilde{\bm{x}}(i) \bigr)^2}{1 - \dfrac{m_{\bm{G}}}{m - |\textsf{E}|}},
\end{equation*}
where we have identified
$\tilde{\bm{x}} = \sum\nolimits_{s=0}^\infty |\bm{G}|^s |\bm{b}|$ and $m_{\bm{G}} = \sum\nolimits_{s=0}^{\infty} \lVert | \bm{G} |^s \rVert_1^2$.
Combining with \cref{eq:combine_me}, we verify 
\begin{equation*}
    \sup_{t \geq 0} \vertiii{\bm{\phi}_{t+1}(\bm{x}_{t}) - \bm{x}_t}^2
    \leq \frac{2}{m - m_{\bm{G}} - |\textsf{E}|} \Bigl(\sum\nolimits_{i \notin \textsf{E}} \tilde{\bm{x}}(i) \Bigr)^2.
\end{equation*}
Last, we optimize over the set $\textsf{E}$ to complete the proof.
\end{proof}

Now, we are ready to bound the RSRI variance and complete the proof of \cref{thm:extended}.

\begin{proposition}[Variance bounds] \label{prop:variance_bound}
Suppose RSRI is applied to an $n \times n$ linear system $\bm{A} \bm{x} = \bm{b}$ and set $\bm{G} = \mathbf{I} - \bm{A}$.
If $\bm{G}$ is a strict $1$-norm contraction, $\lVert \bm{G} \rVert_1 < 1$, then RSRI returns an estimate $\overline{\bm{x}}_t$ with variance bounded by
\begin{equation*}
    \vertiii{ \bm{A} \overline{\bm{x}}_t - \bm{A} \,\mathbb{E} \bigl[ \overline{\bm{x}}_t \bigr] }^2
    \leq \frac{8 t}{(t - t_{\rm b})^2} \cdot 
    \frac{1}{m} \biggl(\frac{\lVert \bm{b} \rVert_1}{1 - \lVert \bm{G} \rVert_1}\biggr)^2.
\end{equation*}
Alternately, if $m \geq m_{\bm{G}} = \sum\nolimits_{s=0}^\infty \lVert | \bm{G} |^s \rVert_1^2$, the variance is bounded by
\begin{equation*}
    \vertiii{ \bm{A} \overline{\bm{x}}_t - \bm{A} \,\mathbb{E} \bigl[ \overline{\bm{x}}_t \bigr] }^2
    \leq 
    \frac{8 t \sup_{s \geq 0} \lVert \bm{G}^s \rVert_1^2}{(t - t_{\rm b})^2}
    \cdot \min_{0 \leq i \leq m - m_{\bm{G}}} \frac{1}{m - m_{\bm{G}} - i} 
    \biggl(\sum\nolimits_{j = i + 1}^n \tilde{\bm{x}}^{\downarrow}(i) \biggr)^2.
\end{equation*}
Here, $\tilde{\bm{x}} \in \mathbb{R}^n$ is the solution to the regularized linear system
$\tilde{\bm{A}} \tilde{\bm{x}} = \tilde{\bm{b}}$, where $\tilde{\bm{A}} = \mathbf{I} - |\bm{G}|$, $\tilde{\bm{b}} = |\bm{b}|$, and $\tilde{\bm{x}}^{\downarrow} \in \mathbb{R}^n$ is the decreasing rearrangement of $\tilde{\bm{x}}$.
\end{proposition}
\begin{proof}
Fix $\bm{u} \in \mathbb{C}^n$ and introduce the martingale
\begin{equation*}
    m_s 
    = \mathbb{E}\bigl[\bm{u}^\ast \bm{A} \overline{\bm{x}}_{t} \,\big|\, \bm{x}_0, \ldots, \bm{x}_s \bigr].
\end{equation*}
A brief calculation shows that the martingale differences are given by
\begin{align*}
    m_s - m_{s-1}
    &= \frac{1}{t - t_{\rm b}} \bm{u}^\ast (\mathbf{I} - \bm{G}) \sum\nolimits_{r=s \vee t_{\rm b}}^{t - 1}
    \bigl(\mathbb{E}\bigl[\bm{x}_r \,\big|\, \bm{x}_s \bigr] 
    - \mathbb{E}\bigl[\bm{x}_r \,\big|\, \bm{x}_{s-1}\bigr]\bigr) \\
    &= \frac{1}{t - t_{\rm b}} \bm{u}^\ast (\mathbf{I} - \bm{G}) \sum\nolimits_{r=s \vee t_{\rm b}}^{t - 1}
    \bm{G}^{r-s + 1} \bigl[ \bm{\phi}_s (\bm{x}_{s-1} ) - \bm{x}_{s-1}\bigr] \\
    &= \frac{1}{t - t_{\rm b}} \bm{u}^\ast \bigl(\bm{G}^{(t_{\rm b} - s + 1) \vee 1} - \bm{G}^{t - s + 1}\bigr) \bigl[ \bm{\phi}_s (\bm{x}_{s-1} ) - \bm{x}_{s-1}\bigr]
\end{align*}
for $1 \leq s \leq t$.
Moreover,
$\operatorname{Var} \bigl[\bm{u}^\ast \bm{A} \overline{\bm{x}}_t \bigr]$ is given by the martingale variance formula
\begin{equation*}
    \operatorname{Var} \bigl[\bm{f}^\ast \bm{A} \overline{\bm{x}}_t\bigr] 
    = \operatorname{Var} \bigl[ m_t \bigr]
    = \sum\nolimits_{s=1}^t \mathbb{E}\bigl| m_s - m_{s-1} \bigr|^2.
\end{equation*}
We bound the variance using the following upper bounds:
\begin{align*}
    & \vertiii{ \bm{A} \overline{\bm{x}}_t - \bm{A} \,\mathbb{E}\bigl[\overline{\bm{x}}_t\bigr] }^2 
    = \max_{\lVert \bm{u} \rVert_{\infty} \leq 1} \operatorname{Var} \bigl[\bm{u}^\ast \bm{A} \overline{\bm{x}}_t \bigr] 
    \\
    &\leq \frac{1}{(t - t_{\rm b})^2} \sum\nolimits_{s=0}^{t - 1} \max_{\lVert \bm{u} \rVert_{\infty} \leq 1} \mathbb{E} \bigl|
    \bm{u}^\ast \bigl(\bm{G}^{(t_{\rm b} - s) \vee 1} - \bm{G}^{t - s}\bigr) \bigl[\bm{\phi}_{s+1}(\bm{x}_s) - \bm{x}_s \bigr]
    \bigr|^2 \\
    &\leq \frac{1}{(t - t_{\rm b})^2} \sum\nolimits_{s=0}^{t - 1} 
    \lVert \bm{G}^{(t_{\rm b} - s) \vee 1} - \bm{G}^{t - s} \rVert_1^2
    \cdot \max_{\lVert \bm{u} \rVert_{\infty} \leq 1}
    \mathbb{E} \bigl|
    \bm{u}^\ast \bigl[\bm{\phi}_{s+1}(\bm{x}_s) - \bm{x}_s \bigr]
    \bigr|^2 \\
    &= \frac{1}{(t - t_{\rm b})^2} \sum\nolimits_{s=0}^{t - 1}
    \lVert \bm{G}^{(t_{\rm b} - s) \vee 1} - \bm{G}^{t - s} \rVert_1^2 \,
    \vertiii{ \bm{\phi}_{s+1}(\bm{x}_s) - \bm{x}_s }^2.
\end{align*}
Last, we use the fact that
\begin{equation*}
    \sum\nolimits_{s = 0}^{t - 1}
    \lVert \bm{G}^{(t_{\rm b} - s) \vee 1} - \bm{G}^{t - s} \rVert_1^2
    \leq 4 t \sup_{s \geq 0}
    \lVert \bm{G}^s \rVert_1^2
\end{equation*}
to establish the general variance formula
\begin{equation*}
    \vertiii{ \bm{A} \overline{\bm{x}}_t - \bm{A} \,\mathbb{E}\bigl[\overline{\bm{x}}_t\bigr] }^2  
    \leq 
    \frac{4 t \sup_{s \geq 0}
    \lVert \bm{G}^s \rVert_1^2}{(t - t_{\rm b})^2}
    \cdot \sup_{s \geq 0} \vertiii{ \bm{\phi}_{s+1} ( \bm{x}_{s} ) - \bm{x}_s }^2.
\end{equation*}
We apply the bounds on $\sup_{s \geq 0} \vertiii{ \bm{\phi}_{s+1} ( \bm{x}_{s} ) - \bm{x}_s }^2$ from \cref{prop:strict,prop:large} to complete the proof.
\end{proof}

\section*{Conclusion}

We have introduced a new algorithm called ``randomly sparsified Richardson iteration'' or ``RSRI'' (\cref{alg:RSRI_general}) for solving $n \times n$ linear systems of equations $\bm{A} \bm{x} = \bm{b}$.
RSRI can be applied to high-dimensional systems with $n \geq 10^9$.
At each iteration, the algorithm only needs to evaluate a random subset of $m$ columns, where $m$ is a parameter specified by the user.
Therefore, RSRI only requires $\mathcal{O}(mN)$ work per iteration if $\bm{A}$ and $\bm{b}$ are dense, or $\mathcal{O}(m q)$ work per iteration if $\bm{A}$ and $\bm{b}$ are sparse with no more than $q$ nonzero entries per column.
Because of this scaling, RSRI can efficiently generate sparse approximations to the solution vector for problems so large that the exact solution cannot be stored as a dense vector on a computer. 

RSRI is an extension of the FRI framework \cite{lim2017fast,greene2019beyond,greene2020improved,greene2022full,greene2022approximating} for speeding up deterministic fixed-point iterations with random sparsification.
In this paper, we have extended FRI for the first time to handle linear systems $\bm{A} \bm{x} = \bm{b}$, and
we have proved that RSRI achieves faster-than-$1/\sqrt{m}$ convergence.
Proving such a result has been a significant obstacle in the mathematical understanding of FRI, and our analysis will serve as the foundation for future algorithmic and mathematical developments. In particular, extending the results in this paper to FRI methods for eigenproblems~\cite{greene2022approximating}, remains an outstanding challenge.



\section*{Acknowledgments}
We would like to acknowledge Bixing Qiao, who helped us investigate classical Monte Carlo schemes for linear systems during his Master's thesis at New York University in 2020. We would also like to thank Lek-Heng Lim who helped to instigate this work, as well as Tyler Chen, Christopher Musco, Kevin Miller, and Ethan N. Epperly who provided valuable comments on an earlier draft.
Last, we would like to acknowledge Jackie Lok, who alerted us to an error in an earlier draft.

\appendix

\section{Derivation of SGD error bounds} \label{sec:discussion}
This section derives the error bounds for randomized coordinate descent and randomized Kaczmarz in \cref{sec:kaczmarz}.

At each iteration, randomized Kaczmarz chooses an index set $\textsf{S} \subseteq \{1, \ldots, n\}$ 
by sampling with replacement from the nonuniform probability distribution $p_j = \lVert \bm{A}(j, \cdot) \rVert^2 / \lVert \bm{A} \rVert_{\rm F}^2$.
Then it updates the iterate according to
\begin{align*}
    \bm{x}_s &= \bm{x}_{s-1} - \frac{\alpha \lVert \bm{A} \rVert_{\rm F}^2}{|\textsf{S}|} \sum\nolimits_{j \in \textsf{S}} \frac{\bm{A}(j, \cdot) \bm{x}_{s-1} - \bm{b}(j)}{\lVert \bm{A}(j, \cdot)^* \rVert^2} \bm{A}(j, \cdot)^* \\
    &= \bm{x}_{s-1} - \frac{\alpha \lVert \bm{A} \rVert_{\rm F}^2}{|\textsf{S}|} \sum\nolimits_{j \in \textsf{S}} \frac{\bm{A}(j, \cdot)^* \bm{A}(j, \cdot)}{\lVert \bm{A}(j, \cdot)^* \rVert^2} (\bm{x}_{s-1} - \bm{x}_\star),
\end{align*}
where the equality holds because $\bm{A}(j, \cdot) \bm{x}_\star = \bm{b}(j)$ for $\bm{x}_\star = \bm{A}^{-1} \bm{b}$.

The conditional expectation of each iterate is given by
\begin{align*}
    \mathbb{E}\bigl[\bm{x}_s - \bm{x}_\star \,|\, \bm{x}_{s-1} \bigr]
    &= \biggl( \mathbf{I} - \frac{\alpha \lVert \bm{A} \rVert_{\rm F}^2}{|\textsf{S}|}\, \mathbb{E}\biggl[\sum\nolimits_{j \in \textsf{S}} \frac{\bm{A}(j, \cdot)^* \bm{A}(j, \cdot)}{\lVert \bm{A}(j, \cdot)^* \rVert^2} \biggr] \biggr) (\bm{x}_{s-1} - \bm{x}_\star) \\
    &= \bigl( \mathbf{I} - \alpha \bm{A}^* \bm{A} \bigr) (\bm{x}_{s-1} - \bm{x}_\star),
\end{align*}
because the expectation of $\sum_{j \in \textsf{S}} \bm{A}(j, \cdot)^* \bm{A}(j, \cdot) / \lVert \bm{A}(j, \cdot) \rVert^2$ is $|\textsf{S}| \bm{A}^* \bm{A} / \lVert \bm{A} \rVert_{\rm F}^2$.

Since the indices $j_1, \ldots, j_{|\textsf{S}|} \in \textsf{S}$ are independent and identically distributed, the trace of the conditional covariance matrix satisfies
\begin{align*}
    &\operatorname{tr} \operatorname{Cov}\bigl[\bm{x}_s - \bm{x}_\star \,|\ \bm{x}_{s-1} \bigr] \\
    &= \frac{\alpha^2 \lVert \bm{A} \rVert_{\rm F}^4}{|\textsf{S}|} 
    \operatorname{tr}
    \operatorname{Cov}\biggl[\frac{\bm{A}(j_1, \cdot)^* \bm{A}(j_1, \cdot)}{\lVert \bm{A}(j_1, \cdot)^* \rVert^2} (\bm{x}_{s-1} - \bm{x}_\star) \,\bigg|\ \bm{x}_{s-1} \biggr] \\
    &= \frac{\alpha^2 \lVert \bm{A} \rVert_{\rm F}^4}{|\textsf{S}|} 
    (\bm{x}_{s-1} - \bm{x}_\star)^*
    \biggl(
    \mathbb{E}\biggl[\frac{\bm{A}(j_1, \cdot)^* \bm{A}(j_1, \cdot)}{\lVert \bm{A}(j_1, \cdot)^* \rVert^2}\biggr]
    - \mathbb{E}\biggl[\frac{\bm{A}(j_1, \cdot)^* \bm{A}(j_1, \cdot)}{\lVert \bm{A}(j_1, \cdot)^* \rVert^2}\biggr]^2
    \biggr)
    (\bm{x}_{s-1} - \bm{x}_\star)
    \\
    &= \frac{\alpha^2 \lVert \bm{A} \rVert_{\rm F}^4}{|\textsf{S}|} 
    (\bm{x}_{s-1} - \bm{x}_\star)^*
    \biggl[
    \frac{\bm{A}^* \bm{A}}{\lVert \bm{A} \rVert_{\rm F}^2}
    - \biggl(\frac{\bm{A}^* \bm{A}}{\lVert \bm{A} \rVert_{\rm F}^2} \biggr)^2
    \biggr]
    (\bm{x}_{s-1} - \bm{x}_\star).
\end{align*}
As the batch size $|\textsf{S}|$ increases, the trace of the conditional variance decreases at the Monte Carlo rate $\sim 1 / |\textsf{S}|$.

Last, we use the conditional bias and covariance formulas to derive the conditional mean square error, according to a vector-valued version of the bias-variance decomposition:
\begin{align*}
    \mathbb{E}\bigl[\lVert \bm{x}_s - \bm{x}_\star \rVert^2 \,|\, \bm{x}_{s-1} \bigr]
    &= 
    \lVert \mathbb{E}\bigl[\bm{x}_s - \bm{x}_\star \,|\, \bm{x}_{s-1} \bigr] \rVert^2 + \operatorname{tr} \operatorname{Cov}\bigl[\bm{x}_s - \bm{x}_\star \,|\ \bm{x}_{s-1} \bigr] \\
    & \begin{aligned}
    &= \lVert \bm{x}_{s-1} - \bm{x}_\star \rVert^2
    + \biggl[\frac{\alpha^2 \lVert \bm{A} \rVert_{\rm F}^2}{|\textsf{S}|} - 2 \alpha \biggr]
    \lVert \bm{A} (\bm{x}_{s-1} - \bm{x}_\star) \rVert^2 \\
    & \quad + \alpha^2 \frac{|\textsf{S}| - 1}{|\textsf{S}|} \lVert \bm{A}^* \bm{A}(\bm{x}_{s-1} - \bm{x}_\star) \rVert^2
    \end{aligned}.
\end{align*}
This establishes the mean square error \cref{eq:rk_takedown}.

Next, to optimize the variance of randomized Kaczmarz, introduce an eigenvalue decomposition $\bm{A}^* \bm{A} = \bm{V}^* \bm{\Lambda} \bm{V}$ where $\bm{V} = \begin{bmatrix} \bm{v}_1 & \cdots & \bm{v}_n \end{bmatrix}$
and $\bm{\Lambda} = \operatorname{diag}(\lambda_1, \ldots, \lambda_n)$ for $\lambda_1 \geq \cdots \geq \lambda_n$.
Observe that
\begin{align*}
    \bm{x}_{s-1} - \bm{x}_\star
    &= \sum\nolimits_{i=1}^n \bm{v}_i \bm{v}_i^* (\bm{x}_{s-1} - \bm{x}_\star), \\    \lVert \bm{x}_{s-1} - \bm{x}_\star \rVert^2
    &= \sum\nolimits_{i=1}^n |\bm{v}_i^* (\bm{x}_{s-1} - \bm{x}_\star)|^2, \\
    \lVert \bm{A} (\bm{x}_{s-1} - \bm{x}_\star) \rVert^2
    &= \sum\nolimits_{i=1}^n \lambda_i |\bm{v}_i^* (\bm{x}_{s-1} - \bm{x}_\star)|^2, \\
    \lVert \bm{A}^* \bm{A} (\bm{x}_{s-1} - \bm{x}_\star) \rVert^2
    &= \sum\nolimits_{i=1}^n \lambda_i^2 |\bm{v}_i^* (\bm{x}_{s-1} - \bm{x}_\star)|^2.
\end{align*}
We define eigenvector overlap terms $\xi_i = \mathbb{E}|\bm{v}_i^* (\bm{x}_{s-1} - \bm{x}_\star)|^2$ for $i = 1, \ldots, n$, so
the expression \cref{eq:rk_takedown} for the mean square error yields
\begin{multline}
\label{eq:optimize_me}
    \mathbb{E}\bigl\lVert \bm{x}_s - \bm{x}_\star \bigr\rVert^2
    = \sum\nolimits_{i=1}^n \xi_i f(\alpha, \lambda_i), \\
    \text{for } f(\alpha, \lambda) = 1 + \biggl[\frac{\alpha^2 \lVert \bm{A} \rVert_{\rm F}^2}{|\textsf{S}|} - 2 \alpha \biggr] \lambda + \alpha^2 \frac{|\textsf{S}| - 1}{|\textsf{S}|} \lambda^2.
\end{multline}
We will proceed to optimize this error expression
by considering the function $f(\alpha, \lambda)$.

For any $\alpha > 0$, the function $\lambda \mapsto f(\alpha, \lambda)$ is a convex quadratic, and it achieves its maximum value for $\lambda \in [\lambda_n, \lambda_1]$ at one of the endpoints, $\lambda \in \{\lambda_1, \lambda_n\}$.
Further observe that $f(\alpha, \lambda_1) \leq f(\alpha, \lambda_n)$
holds if
\begin{equation*}
    f(\alpha, \lambda_1) - f(\alpha, \lambda_n) = \biggl[\frac{\alpha^2 \lVert \bm{A} \rVert_{\rm F}^2}{|\textsf{S}|} - 2 \alpha \biggr] (\lambda_1 - \lambda_n) + \alpha^2 \frac{|\textsf{S}| - 1}{|\textsf{S}|} (\lambda_1^2 - \lambda_n^2) \leq 0,
\end{equation*}
which is equivalent to
\begin{equation}
\label{eq:condition}
    \alpha \leq \frac{2 |\textsf{S}|}{\lVert \bm{A} \rVert_{\rm F}^2 + (|\textsf{S}| - 1) (\lambda_1 + \lambda_n)}.
\end{equation}
Hence, when the step size $\alpha$ is sufficiently small, the worst-case error is achieved for the smallest eigenvalue $\lambda = \lambda_n$.

We can optimize the worst-case error with $\lambda = \lambda_n$ by using the step size
\begin{equation*}
    \alpha = \operatornamewithlimits{argmin}_{\beta > 0} f(\lambda_n, \beta) = \frac{|\textsf{S}|}{\lVert \bm{A} \rVert_{F}^2 + (|\textsf{S}| - 1)\lambda_n}.
\end{equation*}
Observe that this step size satisfies the smallness criterion \eqref{eq:condition} if
\begin{equation*}
    (|\textsf{S}| - 1) (\lambda_1 - \lambda_n) \leq \lVert \bm{A} \rVert_{F}^2.
\end{equation*}
Under the smallness criterion, we have shown
\begin{equation*}
    f(\alpha, \lambda) \leq f(\alpha, \lambda_n) = 1 - \frac{|\textsf{S}| \lambda_n}{\lVert \bm{A} \rVert_{F}^2 + (|\textsf{S}| - 1)\lambda_n}
    \quad \text{for each } \lambda \in [\lambda_n, \lambda_1].
\end{equation*}
The expression \eqref{eq:optimize_me} for the randomized Kaczmarz error then yields
\begin{equation*}
    \mathbb{E}\lVert \bm{x}_s - \bm{x}_\star \rVert^2
    \leq \biggl[1 - \frac{|\textsf{S}| \lambda_n}{\lVert \bm{A} \rVert_{F}^2 + (|\textsf{S}| - 1)\lambda_n}\biggr] \mathbb{E}\lVert \bm{x}_{s-1} - \bm{x}_\star \rVert^2
\end{equation*}
This completes the analysis of randomized Kaczmarz.

Last, we transfer the error bounds from randomized Kaczmarz to randomized coordinate descent.
Using randomized coordinate descent to solve $\bm{A} \bm{x} = \bm{b}$
produces the same distribution of iterates as using randomized Kaczmarz to solve $\bm{A}^{1/2} \bm{y} = \bm{b}$ and then setting $\bm{x} = \bm{A}^{-1/2} \bm{y}$ \cite[Sec.~5.1]{derezinski2025randomizedkaczmarzmethodsbeyondkrylov}.
We have already derived error bounds for iterates $\bm{y}_s$ when randomized Kaczmarz is applied to $\bm{A}^{1/2} \bm{y} = \bm{b}$.
We obtain corresponding error bounds for randomized coordinate descent by substituting $\bm{y}_s = \bm{A}^{1/2} \bm{x}_s$ in the right places.

\bibliographystyle{siamplain}
\bibliography{references}
\end{document}

%% file: ex_shared.tex
\usepackage{lipsum}
\usepackage{amsfonts}
\usepackage{graphicx}
\usepackage{epstopdf}
\usepackage{enumitem}
\usepackage{xurl}
\usepackage{myalgpseudocode}
\renewcommand{\algorithmicrequire}{\textbf{Input:}}
\renewcommand{\algorithmicensure}{\textbf{Output:}}
\newcommand{\smalliii}[1]{{\vert\kern-0.25ex\vert\kern-0.25ex\vert #1 
    \vert\kern-0.25ex\vert\kern-0.25ex\vert}}
\newcommand{\vertiii}[1]{{\bigl\vert\kern-0.25ex\bigl\vert\kern-0.25ex\bigl\vert #1 
    \bigr\vert\kern-0.25ex\bigr\vert\kern-0.25ex\bigr\vert}}
\ifpdf
  \DeclareGraphicsExtensions{.eps,.pdf,.png,.jpg}
\else
  \DeclareGraphicsExtensions{.eps}
\fi

\newsiamremark{remark}{Remark}
\newsiamremark{hypothesis}{Hypothesis}
\crefname{hypothesis}{Hypothesis}{Hypotheses}
\newsiamthm{claim}{Claim}

\headers{Randomly sparsified Richardson iteration}{J. Weare and R. J. Webber}

\title{Randomly sparsified Richardson iteration: \\ A dimension-independent sparse linear solver
\thanks{
\textbf{Funding:}
JW acknowledges support from the Advanced Scientific
Computing Research Program within the DOE Office of Science through award DE-SC0020427 and from the National Science Foundation through award DMS-2054306.
}}

\author{Jonathan Weare\thanks{Courant Institute of Mathematical Sciences, New York University, New York, NY 10012
(\email{weare@nyu.edu})}
\and Robert J. Webber\thanks{Department of Mathematics, University of California San Diego, La Jolla, CA 92093 (\email{rwebber@ucsd.edu})}}

\usepackage{amsopn}
\usepackage{dsfont,bm,mathtools}

\makeatletter
\newcommand*{\addFileDependency}[1]{
  \typeout{(#1)}
  \@addtofilelist{#1}
  \IfFileExists{#1}{}{\typeout{No file #1.}}
}
\makeatother


%% file: ex_article.bbl
\begin{thebibliography}{10}

\bibitem{agirre2009personalizing}
{\sc E.~Agirre and A.~Soroa}, {\em Personalizing {PageRank} for word sense
  disambiguation}, in Proceedings of the 12th Conference of the European
  Chapter of the Association for Computational Linguistics, 2009, p.~33–41,
  \url{https://dl.acm.org/doi/10.5555/1609067.1609070}.

\bibitem{albert1999diameter}
{\sc R.~Albert, H.~Jeong, and A.-L. Barab{\'{a}}si}, {\em Diameter of the
  {World-Wide Web}}, Nature, 401 (1999), pp.~130--131,
  \url{https://doi.org/10.1038/43601}.

\bibitem{anand2019amazon}
{\sc S.~Anand}, {\em Kaggle datasets: {A}mazon product reviews}, 2019,
  \url{https://www.kaggle.com/datasets/saurav9786/amazon-product-reviews}.
\newblock Version 1.

\bibitem{andersen2007local}
{\sc R.~Andersen, C.~Borgs, J.~Chayes, J.~Hopcraft, V.~S. Mirrokni, and S.-H.
  Teng}, {\em Local computation of {PageRank} contributions}, in 5th
  International Workshop on Algorithms and Models for the Web-Graph, 2007,
  \url{https://doi.org/10.1007/978-3-540-77004-6_12}.

\bibitem{andoni2018solving}
{\sc A.~Andoni, R.~Krauthgamer, and Y.~Pogrow}, {\em On solving linear systems
  in sublinear time}, 2018, \url{https://arxiv.org/abs/1809.02995}.

\bibitem{avrachenkov2007monte}
{\sc K.~Avrachenkov, N.~Litvak, D.~Nemirovsky, and N.~Osipova}, {\em Monte
  {C}arlo methods in {PageRank} computation: {W}hen one iteration is
  sufficient}, SIAM Journal on Numerical Analysis, 45 (2007), pp.~890--904,
  \url{https://doi.org/10.1137/050643799}.

\bibitem{bauer1958monte}
{\sc W.~F. Bauer}, {\em The {M}onte {C}arlo method}, Journal of the Society for
  Industrial and Applied Mathematics, 6 (1958), pp.~438--451,
  \url{https://doi.org/10.1137/0106028}.

\bibitem{benzi2017analysis}
{\sc M.~Benzi, T.~M. Evans, S.~P. Hamilton, M.~Lupo~Pasini, and S.~R.
  Slattery}, {\em Analysis of {M}onte {C}arlo accelerated iterative methods for
  sparse linear systems}, Numerical Linear Algebra with Applications, 24
  (2017), \url{https://doi.org/10.1002/nla.2088}.

\bibitem{berkhin2006bookmark}
{\sc P.~Berkhin}, {\em Bookmark-coloring algorithm for personalized {PageRank}
  computing}, Internet Mathematics, 3 (2006), pp.~41--62,
  \url{https://doi.org/10.1080/15427951.2006.10129116}.

\bibitem{booth2009fermion}
{\sc G.~H. Booth, A.~J.~W. Thom, and A.~Alavi}, {\em Fermion {M}onte {C}arlo
  without fixed nodes: {A} game of life, death, and annihilation in slater
  determinant space}, The Journal of Chemical Physics, 131 (2009), p.~054106,
  \url{https://doi.org/10.1063/1.3193710}.

\bibitem{borgs2014multiscale}
{\sc C.~Borgs, M.~Brautbar, J.~Chayes, and S.-H. Teng}, {\em Multiscale matrix
  sampling and sublinear-time {PageRank} computation}, Internet Mathematics, 10
  (2014), pp.~20--48, \url{https://doi.org/10.1080/15427951.2013.802752}.

\bibitem{chen2019active}
{\sc X.~Chen and E.~Price}, {\em Active regression via linear-sample
  sparsification}, in Proceedings of the Thirty-Second Conference on Learning
  Theory, 2019, \url{https://proceedings.mlr.press/v99/chen19a.html}.

\bibitem{cleland2010survival}
{\sc D.~Cleland, G.~H. Booth, and A.~Alavi}, {\em Communications: {S}urvival of
  the fittest: {A}ccelerating convergence in full configuration-interaction
  quantum {M}onte {C}arlo}, The Journal of Chemical Physics, 132 (2010),
  p.~041103, \url{https://doi.org/10.1063/1.3302277}.

\bibitem{cohen2001adaptive}
{\sc A.~Cohen, W.~Dahmen, and R.~Devore}, {\em Adaptive wavelet methods for
  elliptic operator equations: {C}onvergence rates}, Mathematics of
  Computation, 70 (2001), pp.~27--75,
  \url{http://www.jstor.org/stable/2698924}.

\bibitem{curtiss1953monte}
{\sc J.~H. Curtiss}, {\em {“Monte Carlo”} methods for the iteration of
  linear operators}, Journal of Mathematics and Physics, 32 (1953),
  pp.~209--232, \url{https://doi.org/10.1002/sapm1953321209}.

\bibitem{derezinski2025randomizedkaczmarzmethodsbeyondkrylov}
{\sc M.~Dereziński, D.~Needell, E.~Rebrova, and J.~Yang}, {\em Randomized
  kaczmarz methods with beyond-krylov convergence}, 2025,
  \url{https://arxiv.org/abs/2501.11673}.

\bibitem{deville1998unequal}
{\sc J.-C. Deville and Y.~Tille}, {\em Unequal probability sampling without
  replacement through a splitting method}, Biometrika, 85 (1998), pp.~89--101,
  \url{http://www.jstor.org/stable/2337311}.

\bibitem{dimov1991minimization}
{\sc I.~Dimov}, {\em Minimization of the probable error for some {M}onte
  {C}arlo methods}, in Proceedings of the International Conference on
  Mathematical Modeling and Scientific Computation, 1991, pp.~159--170.

\bibitem{edmundson1953monte}
{\sc H.~P. Edmundson}, {\em Monte {C}arlo matrix inversion and recurrent
  events}, Mathematical Tables and Other Aids to Computation, 7 (1953),
  pp.~18--21, \url{https://doi.org/10.2307/2002564}.

\bibitem{epperly2025randomizedkaczmarztailaveraging}
{\sc E.~N. Epperly, G.~Goldshlager, and R.~J. Webber}, {\em Randomized kaczmarz
  with tail averaging}, 2025, \url{https://arxiv.org/abs/2411.19877}.

\bibitem{evans2014monte}
{\sc T.~M. Evans, S.~W. Mosher, S.~R. Slattery, and S.~P. Hamilton}, {\em A
  {M}onte {C}arlo synthetic-acceleration method for solving the thermal
  radiation diffusion equation}, Journal of Computational Physics, 258 (2014),
  pp.~338--358, \url{https://doi.org/10.1016/j.jcp.2013.10.043}.

\bibitem{floyd64algorithm}
{\sc R.~W. Floyd}, {\em Algorithm 245: {T}reesort}, Communications of the ACM,
  7 (1964), p.~701, \url{https://doi.org/10.1145/355588.365103}.

\bibitem{fogaras2005towards}
{\sc D.~Fogaras, B.~Rácz, K.~Csalogány, and T.~Sarlós}, {\em Towards scaling
  fully personalized {PageRank}: {A}lgorithms, lower bounds, and experiments},
  Internet Mathematics, 2 (2005), pp.~333--358,
  \url{https://doi.org/10.1080/15427951.2005.10129104}.

\bibitem{forsythe1950matrix}
{\sc G.~E. Forsythe and R.~A. Leibler}, {\em Matrix inversion by a {M}onte
  {C}arlo method}, Mathematical Tables and Other Aids to Computation, 4 (1950),
  pp.~127--129, \url{http://www.jstor.org/stable/2002508}.

\bibitem{garrigos2024handbookconvergencetheoremsstochastic}
{\sc G.~Garrigos and R.~M. Gower}, {\em Handbook of convergence theorems for
  (stochastic) gradient methods}, 2024, \url{https://arxiv.org/abs/2301.11235}.

\bibitem{gilyen2022improved}
{\sc A.~Gily{\'{e}}n, Z.~Song, and E.~Tang}, {\em An improved quantum-inspired
  algorithm for linear regression}, {Quantum}, 6 (2022), p.~754,
  \url{https://doi.org/10.22331/q-2022-06-30-754}.

\bibitem{gleich2006approximating}
{\sc D.~Gleich and M.~Polito}, {\em Approximating personalized {PageRank} with
  minimal use of web graph data}, Internet Mathematics, 3 (2006),
  \url{https://doi.org/10.1080/15427951.2006.10129128}.

\bibitem{gleich2015pagerank}
{\sc D.~F. Gleich}, {\em Page{R}ank beyond the web}, SIAM Review, 57 (2015),
  pp.~321--363, \url{https://doi.org/10.1137/140976649}.

\bibitem{gori2007itemrank}
{\sc M.~Gori and A.~Pucci}, {\em {ItemRank}: {A} random-walk based scoring
  algorithm for recommender engines}, in Proceedings of the 20th International
  Joint Conference on Artifical Intelligence, 2007, p.~2766–2771,
  \url{https://dl.acm.org/doi/10.5555/1625275.1625720}.

\bibitem{gower2019sgd}
{\sc R.~M. Gower, N.~Loizou, X.~Qian, A.~Sailanbayev, E.~Shulgin, and
  P.~Richt{\'a}rik}, {\em {SGD}: {G}eneral analysis and improved rates}, in
  Proceedings of the 36th International Conference on Machine Learning, 2019,
  \url{https://proceedings.mlr.press/v97/qian19b.html}.

\bibitem{greene2022approximating}
{\sc S.~M. Greene, R.~J. Webber, T.~C. Berkelbach, and J.~Weare}, {\em
  Approximating matrix eigenvalues by subspace iteration with repeated random
  sparsification}, SIAM Journal on Scientific Computing, 44 (2022),
  pp.~A3067--A3097, \url{https://doi.org/10.1137/21M1422513}.

\bibitem{greene2022full}
{\sc S.~M. Greene, R.~J. Webber, J.~E.~T. Smith, J.~Weare, and T.~C.
  Berkelbach}, {\em Full configuration interaction excited-state energies in
  large active spaces from subspace iteration with repeated random
  sparsification}, Journal of Chemical Theory and Computation, 18 (2022),
  pp.~7218--7232, \url{https://doi.org/10.1021/acs.jctc.2c00435}.

\bibitem{greene2019beyond}
{\sc S.~M. Greene, R.~J. Webber, J.~Weare, and T.~C. Berkelbach}, {\em Beyond
  walkers in stochastic quantum chemistry: Reducing error using fast randomized
  iteration}, Journal of Chemical Theory and Computation, 15 (2019),
  pp.~4834--4850, \url{https://doi.org/10.1021/acs.jctc.9b00422}.

\bibitem{greene2020improved}
{\sc S.~M. Greene, R.~J. Webber, J.~Weare, and T.~C. Berkelbach}, {\em Improved
  fast randomized iteration approach to full configuration interaction},
  Journal of Chemical Theory and Computation, 16 (2020), pp.~5572--5585,
  \url{https://doi.org/10.1021/acs.jctc.0c00437}.

\bibitem{halton1994sequential}
{\sc J.~H. Halton}, {\em Sequential {M}onte {C}arlo techniques for the solution
  of linear systems}, Journal of Scientific Computing, 9 (1994), pp.~213--257,
  \url{https://doi.org/10.1007/bf01578388}.

\bibitem{horn2012matrix}
{\sc R.~A. Horn and C.~R. Johnson}, {\em Matrix Analysis}, Cambridge University
  Press, second~ed., 2012, \url{https://doi.org/10.1017/CBO9781139020411}.

\bibitem{horner1819new}
{\sc W.~G. Horner and D.~Gilbert}, {\em {XXI}. {A} new method of solving
  numerical equations of all orders, by continuous approximation},
  Philosophical Transactions of the Royal Society of London, 109 (1819),
  pp.~308--335, \url{https://doi.org/10.1098/rstl.1819.0023}.

\bibitem{jeh2003scaling}
{\sc G.~Jeh and J.~Widom}, {\em Scaling personalized web search}, in
  Proceedings of the 12th International Conference on World Wide Web, 2003,
  \url{https://doi.org/10.1145/775152.775191}.

\bibitem{ji2013convergence}
{\sc H.~Ji, M.~Mascagni, and Y.~Li}, {\em Convergence analysis of {M}arkov
  chain {M}onte {C}arlo linear solvers using {U}lam--von {N}eumann algorithm},
  SIAM Journal on Numerical Analysis, 51 (2013), pp.~2107--2122,
  \url{https://doi.org/10.1137/130904867}.

\bibitem{lai2009adaptive}
{\sc Y.~Lai}, {\em Adaptive {M}onte {C}arlo methods for matrix equations with
  applications}, Journal of Computational and Applied Mathematics, 231 (2009),
  pp.~705--714, \url{https://doi.org/10.1016/j.cam.2009.04.008}.

\bibitem{leskovec1999notre}
{\sc J.~Leskovec and A.~Krevl}, {\em {SNAP} datasets: {N}otre {D}ame web
  graph}, 1999, \url{http://snap.stanford.edu/data/web-NotreDame.html}.

\bibitem{leventhal2010randomized}
{\sc D.~Leventhal and A.~S. Lewis}, {\em Randomized methods for linear
  constraints: Convergence rates and conditioning}, Mathematics of Operations
  Research, 35 (2010), pp.~641--654,
  \url{https://doi.org/10.1287/moor.1100.0456}.

\bibitem{lim2017fast}
{\sc L.-H. Lim and J.~Weare}, {\em Fast randomized iteration: {D}iffusion
  {M}onte {C}arlo through the lens of numerical linear algebra}, SIAM Review,
  59 (2017), pp.~547--587, \url{https://doi.org/10.1137/15M1040827}.

\bibitem{lu2020full}
{\sc J.~Lu and Z.~Wang}, {\em The full configuration interaction quantum
  {M}onte {C}arlo method through the lens of inexact power iteration}, SIAM
  Journal on Scientific Computing, 42 (2020), pp.~B1--B29,
  \url{https://doi.org/10.1137/18M1166626}.

\bibitem{moorman2020randomized}
{\sc J.~D. Moorman, T.~K. Tu, D.~Molitor, and D.~Needell}, {\em Randomized
  {K}aczmarz with averaging}, BIT Numerical Mathematics, 61 (2020),
  p.~337–359, \url{https://doi.org/10.1007/s10543-020-00824-1}.

\bibitem{needell2015stochastic}
{\sc D.~Needell, N.~Srebro, and R.~Ward}, {\em Stochastic gradient descent,
  weighted sampling, and the randomized {K}aczmarz algorithm}, Mathematical
  Programming, 155 (2015), pp.~549--573,
  \url{https://doi.org/10.1007/s10107-015-0864-7}.

\bibitem{ni2019justifying}
{\sc J.~Ni, J.~Li, and J.~McAuley}, {\em Justifying recommendations using
  distantly-labeled reviews and fine-grained aspects}, in Proceedings of the
  2019 Conference on Empirical Methods in Natural Language Processing and the
  9th International Joint Conference on Natural Language Processing, 2019,
  \url{https://doi.org/10.18653/v1/D19-1018}.

\bibitem{nocedal2006numerical}
{\sc J.~Nocedal and S.~J. Wright}, {\em Numerical Optimization}, Springer,
  second~ed., 2006, \url{https://doi.org/10.1007/978-0-387-40065-5}.

\bibitem{okten2005solving}
{\sc G.~\"{O}kten}, {\em Solving linear equations by {M}onte {C}arlo
  simulation}, SIAM Journal on Scientific Computing, 27 (2005), pp.~511--531,
  \url{https://doi.org/10.1137/04060500X}.

\bibitem{opsahl2011why}
{\sc T.~Opsahl}, {\em Why {A}nchorage is not (that) important: {B}inary ties
  and sample selection}, 2011,
  \url{https://toreopsahl.com/2011/08/12/why-anchorage-is-not-that-important-binary-ties-and-sample-selection/}.

\bibitem{ozdaglar2020asynchronous}
{\sc A.~Ozdaglar, D.~Shah, and C.~L. Yu}, {\em Asynchronous approximation of a
  single component of the solution to a linear system}, IEEE Transactions on
  Network Science and Engineering, 7 (2020), pp.~975--986,
  \url{https://doi.org/10.1109/TNSE.2019.2894990}.

\bibitem{page1999pagerank}
{\sc L.~Page, S.~Brin, R.~Motwani, and T.~Winograd}, {\em The {PageRank}
  citation ranking: {B}ringing order to the web.}, Tech. Report 1999-66,
  Stanford InfoLab, 1999, \url{http://ilpubs.stanford.edu:8090/422/}.

\bibitem{qu2016coordinate}
{\sc Z.~Qu and P.~Richtárik}, {\em Coordinate descent with arbitrary sampling
  i: algorithms and complexity†}, Optimization Methods and Software, 31
  (2016), pp.~829--857, \url{https://doi.org/10.1080/10556788.2016.1190360}.

\bibitem{rathore2025askotchneatsolutionlargescale}
{\sc P.~Rathore, Z.~Frangella, J.~Yang, M.~Dereziński, and M.~Udell}, {\em
  Have askotch: A neat solution for large-scale kernel ridge regression}, 2025,
  \url{https://arxiv.org/abs/2407.10070}.

\bibitem{richardson1911approximation}
{\sc L.~F. Richardson and R.~T. Glazebrook}, {\em Ix. {T}he approximate
  arithmetical solution by finite differences of physical problems involving
  differential equations, with an application to the stresses in a masonry
  dam}, Philosophical Transactions of the Royal Society of London. Series A,
  Containing Papers of a Mathematical or Physical Character, 210 (1911),
  pp.~307--357, \url{https://doi.org/10.1098/rsta.1911.0009}.

\bibitem{rosenblatt1958perceptron}
{\sc F.~Rosenblatt}, {\em The perceptron: A probabilistic model for information
  storage and organization in the brain.}, Psychological Review, 65 (1958),
  p.~386–408, \url{https://doi.org/10.1037/h0042519}.

\bibitem{sadeh1974monte}
{\sc E.~Sadeh and M.~Franklin}, {\em Monte {C}arlo solution of partial
  differential equations by special purpose digital computer}, IEEE
  Transactions on Computers, C-23 (1974), pp.~389--397,
  \url{https://doi.org/10.1109/T-C.1974.223954}.

\bibitem{sarlos2006to}
{\sc T.~Sarl\'{o}s, A.~A. Bencz\'{u}r, K.~Csalog\'{a}ny, D.~Fogaras, and
  B.~R\'{a}cz}, {\em To randomize or not to randomize: {S}pace optimal
  summaries for hyperlink analysis}, in Proceedings of the 15th International
  Conference on World Wide Web, 2006,
  \url{https://doi.org/10.1145/1135777.1135823}.

\bibitem{shao2022faster}
{\sc C.~Shao and A.~Montanaro}, {\em Faster quantum-inspired algorithms for
  solving linear systems}, ACM Transactions on Quantum Computing, 3 (2022),
  \url{https://doi.org/10.1145/3520141}.

\bibitem{shepherd2012full}
{\sc J.~J. Shepherd, G.~Booth, A.~Gr\"uneis, and A.~Alavi}, {\em Full
  configuration interaction perspective on the homogeneous electron gas},
  Physical Review B, 85 (2012), p.~081103,
  \url{https://doi.org/10.1103/PhysRevB.85.081103}.

\bibitem{shyamkumar2016sublinear}
{\sc N.~Shyamkumar, S.~Banerjee, and P.~Lofgren}, {\em Sublinear estimation of
  a single element in sparse linear systems}, in 54th Annual Allerton
  Conference on Communication, Control, and Computing, 2016, pp.~856--860,
  \url{https://doi.org/10.1109/ALLERTON.2016.7852323}.

\bibitem{srinivasan2001distributions}
{\sc A.~Srinivasan}, {\em Distributions on level-sets with applications to
  approximation algorithms}, in Proceedings 42nd IEEE Symposium on Foundations
  of Computer Science, 2001, pp.~588--597.

\bibitem{strohmer2008randomized}
{\sc T.~Strohmer and R.~Vershynin}, {\em A randomized {Kaczmarz} algorithm with
  exponential convergence}, Journal of Fourier Analysis and Applications, 15
  (2008), pp.~262--278, \url{https://doi.org/10.1007/s00041-008-9030-4}.

\bibitem{wang2008monte}
{\sc Q.~Wang, D.~Gleich, A.~Saberi, N.~Etemadi, and P.~Moin}, {\em A {M}onte
  {C}arlo method for solving unsteady adjoint equations}, Journal of
  Computational Physics, 227 (2008), pp.~6184--6205,
  \url{https://doi.org/10.1016/j.jcp.2008.03.006}.

\bibitem{wasow1952note}
{\sc W.~R. Wasow}, {\em A note on the inversion of matrices by random walks},
  Mathematical Tables and Other Aids to Computation, 6 (1952), pp.~78--81,
  \url{http://www.jstor.org/stable/2002546}.

\bibitem{Wis06}
{\sc A.~Wissner-Gross}, {\em Preparation of topical reading lists from the link
  structure of {W}ikipedia}, in Sixth IEEE International Conference on Advanced
  Learning Technologies, 2006, pp.~825--829,
  \url{https://doi.org/10.1109/ICALT.2006.1652568}.

\end{thebibliography}
